\documentclass[12pt]{article}
\usepackage{amsmath}
\usepackage{latexsym}
\usepackage{amssymb}
\newtheorem{thm}{Theorem}[section]
\newtheorem{la}[thm]{Lemma}
\newtheorem{Defn}[thm]{Definition}
\newtheorem{Exa}[thm]{Example}
\newtheorem{Remark}[thm]{Remark}
\newtheorem{prop}[thm]{Proposition}

\newtheorem{Number}[thm]{\!\!}
\newenvironment{defn}{\begin{Defn}\rm}{\end{Defn}}

\newenvironment{rem}{\begin{Remark}\rm}{\end{Remark}}
\newenvironment{numba}{\begin{Number}\rm}{\end{Number}}
\newenvironment{proof}{{\noindent\bf Proof.}}%
                  {\nopagebreak\hspace*{\fill}$\Box$\medskip\par}
\newcommand{\wb}{\overline}
\newcommand{\ve}{\varepsilon}
\newcommand{\wt}{\widetilde}
\newcommand{\impl}{\Rightarrow}
\newcommand{\mto}{\mapsto}

\newcommand{\N}{{\mathbb N}}
\newcommand{\R}{{\mathbb R}}

\newcommand{\cg}{{\mathfrak g}}

\newcommand{\cL}{{\mathcal L}}
\newcommand{\cO}{{\mathcal O}}
\newcommand{\cB}{{\mathcal B}}

\newcommand{\cR}{{\mathcal R}}
\newcommand{\cE}{{\mathcal E}}
\newcommand{\sub}{\subseteq}
\DeclareMathOperator{\id}{id}
\DeclareMathOperator{\Diff}{Diff}

\DeclareMathOperator{\Evol}{Evol}
\DeclareMathOperator{\evol}{evol}

\DeclareMathOperator{\Reach}{Reach}
\DeclareMathOperator{\Fl}{Fl}

\DeclareMathOperator{\conv}{conv}
\DeclareMathOperator{\ex}{ex}
\DeclareMathOperator{\im}{im}

\newcommand{\one}{1}
\DeclareMathOperator{\Ad}{Ad}
\DeclareMathOperator{\pr}{pr}
\begin{document}
\begin{center}
{\Large\bf
Aspects of control theory on\\[2.2mm]
infinite-dimensional Lie groups\\[4.2mm]
and {\boldmath$G$}-manifolds}\\[7mm]
{\bf Helge Gl\"{o}ckner
and Joachim Hilgert}\vspace{3mm}
\end{center}
\begin{abstract}
\hspace*{-5.9mm}We develop aspects of geometric control theory
on
Lie groups~$G$
which may be infinite dimensional,
and
on smooth $G$-manifolds~$M$
modelled on
locally convex
spaces.
As a tool, we discuss existence and uniqueness questions
for differential equations on~$M$ given by time-dependent
fundamental vector fields which are $\cL^1$ in time.
We then discuss the closures of reachable sets in~$M$
for controls in the Lie algebra~$\cg$ of~$G$,
or within a compact convex subset of~$\cg$.
Regularity properties of the Lie group~$G$
play an important role.\vspace{2mm}
\end{abstract}
{\bf Classification:}
22E65 (primary);
28B05,
%Vector-valued set functions, measures and integrals
34A12,
%Initial value problems, existence, uniqueness, continuous dependence
34H05,
%Control problems involving ordinary differential equations
%
46E30,
%Spaces of measurable functions (Lp
%-spaces, Orlicz spaces, Köthe function spaces, Lorentz spaces,...
%
46E40.\\[2.3mm]
%Spaces of vector- and operator-valued functions
%
%
%
\section{\!Introduction and statement of main results}
We lay some foundations for control
theory on smooth $G$-manifolds~$M$,
for~$G$ a Lie group modelled on a locally convex space.
As a starting point, we discuss
local existence and uniqueness
for Carath\'{e}odory solutions
to time-dependent fundamental vector fields on~$M$
with $\cL^1$-time dependence.
Assuming that $G$ is $L^1$-regular,
we obtain results concerning reachable
points.\\[2.3mm]
Ordinary differential equations
in Banach spaces are a classical topic
(see, e.g., \cite{Car},
\cite{Die}, \cite{Lan}).
Beyond normable spaces,
initial value problems
on locally convex spaces
need not have
solutions, and may possess
many solutions (see Example~6.1 and Example~6.2 in~\cite{Mi0},
also \cite[\S2.4]{GaN}).
But for special classes of equations,
specific results are available,
notably when Lie groups come into play.
If $G$ is a Lie group
modelled on a locally convex space,
then $G$ gives rise
to a smooth left action $G\times TG\to TG$,
$(g,v)\mto g.v$ on its tangent bundle
$TG$ via left translation,
$g.v:=T\lambda_g(v)$ with $\lambda_g\colon G\to G$,
$h\mto gh$.
Let $e\in G$ be the neutral element.
If $\gamma\colon [0,1]\to\cg$
is a continuous path in the Lie algebra
$\cg:=T_eG$ of~$G$,
then the initial value problem
\[
\dot{y}(t)=y(t).\gamma(t),\quad y(0)=e
\]
has at most one $C^1$-solution $\eta\colon [0,1]\to G$,
by \cite[Lemma~7.4]{Mil}.\footnote{Smoothness of $\gamma$
and sequential completeness of $\cg$ (which are assumed in loc.\,cit.)
are not used in the proof.
The fact also follows from
Lemmas~2.5.12 and 2.5.4 in~\cite{GaN}.}
If~$\eta$ exists,
it is called the \emph{evolution}
of~$\gamma$ and denoted by $\Evol(\gamma):=\eta$.
Let $k\in\N_0\cup\{\infty\}$.
If $\Evol(\gamma)$ exists for all
$\gamma\in C^k([0,1],\cg)$
and the map
\[
\evol\colon C^k([0,1],\cg)\to G,\quad
\gamma\mto\Evol(\gamma)(1)
\]
is smooth, then~$G$ is called \emph{$C^k$-regular}
(see~\cite{SEM}). The $C^\infty$-regular Lie groups
are also called \emph{regular};
every $C^k$-regular Lie group is regular.
Regularity is a central concept in infinite-dimensional
Lie theory; see \cite{Mil,KM0,KaM,SEM,GaN}, and \cite{Han}
for further information, notably the survey \cite{Nee}.\\[2.3mm]
If $M$ is a smooth manifold modelled
on a locally convex space,
endowed with a smooth right $G$-action
$\sigma\colon M\times G\to M$,
then each $v\in \cg$ determines a smooth vector~field
\[
v_\sharp\colon M\to TM,\;\, x\mto (T_e\sigma(x,\cdot))(v)
\]
on~$M$, the \emph{fundamental vector field} associated with~$v$.
If a continuous path $\gamma\colon [0,1]\to\cg$
admits a $C^1$-evolution $\eta\colon [0,1]\to G$,
it is known that the differential equation
\begin{equation}\label{theodeGmfd}
\dot{y}(t)=\gamma(t)_\sharp(y(t))
\end{equation}
on~$M$ satisfies local existence and uniqueness
of $C^1$-solutions, since
\[
\Fl\colon [0,1]\times [0,1]\times M\to M,\;\,
(t,t_0,y_0)\mto \sigma(y_0,\eta(t_0)^{-1}\eta(t))
\]
is a globally defined $C^1$-flow for~(\ref{theodeGmfd}),
see \cite[Lemma~2.5.13]{GaN}.\\[2.3mm]
As our starting point,
we establish analogous existence and uniqueness
results for Carath\'{e}odory
solutions, when fundamental vector fields
are $\cL^1$ in time.
For simplicity of the formulation,
we restrict attention to Lie groups
and manifolds modelled
on Fr\'{e}chet spaces
in the remainder of the introduction,
and also in much of Sections~\ref{sec-prels} through \ref{sec-approachable}
and \ref{sec-Lq};
in Section~\ref{sec-beyond},
we explain the necessary changes allowing us
to extend the results to
Lie groups and manifolds
modelled on more general locally convex spaces.
(The Fr\'{e}chet property is also irrelevant in Sections~\ref{sec-han}
and \ref{sec-C0}).\\[2.3mm]
For real numbers $a<b$ and a Fr\'{e}chet
space~$E$, a function $\gamma\colon [a,b]\to E$
is called $\cL^1$ if it is measurable
with respect to the Borel $\sigma$-algebras
on domain and range,
the image $\gamma([a,b])$ contains a countable
dense subset and
\begin{equation}\label{theL1}
\|\gamma\|_{\cL^1,q}:=\int_{[a,b]}q(\gamma(t))\, d\lambda_1(t)<\infty
\end{equation}
for each continuous seminorm $q$ on~$E$
(where $\lambda_1$ denotes Lebesgue-Borel measure).
Let $N\sub\cL^1([a,b],E)$ be the set of $\cL^1$-functions
which vanish almost everywhere (with respect to
$\lambda_1$), and write $[\gamma]:=\gamma+N$
for $\gamma\in \cL^1([a,b],E)$.
Then $L^1([a,b],E):=\cL^1([a,b], E)/N$
is a Fr\'{e}chet space with respect
to the locally convex vector topology given
by the seminorms $\|\cdot\|_{L^1,q}$
defined via $\|[\gamma]\|_{L^1,q}:=\|\gamma\|_{\cL^1,q}$
(see \cite[Lemma~1.19]{MeR}).
Following~\cite[\S3]{MeR},
% (with ${\mathcal E}:=L^1$),
we say that a function $\eta\colon [a,b]\to E$
is \emph{absolutely continuous}\footnote{If $E$ is a Banach space,
copying the classical $\ve$-$\delta$-definition
of absolute continuity of scalar-valued functions (as in \cite[Definition~7.17]{Rud}),
one obtains a more general concept of absolutely
continuous functions which need not be absolutely
continuous in the above more limited sense
(see \cite{Boc}) unless~$E$
is sufficiently nice (e.g., a Hilbert space),
see \cite{Cla}.}
if there exists $\gamma\in \cL^1([a,b],E)$
such that
\begin{equation}\label{gammaprimi}
\eta(t)=\eta(a)+\int_a^t\gamma(s)\,ds
\quad\mbox{for all $t\in [a,b]$,}
\end{equation}
writing
$\int_a^t \gamma(s)\,ds$ for the weak
integral\footnote{The same notation will denote
weak integrals with respect to Lebesgue measure.}
$\int_{[a,t]}\gamma(s)\,d\lambda_1(s)\in E$.
Then
\[
\gamma(t)=\eta'(t)\mbox{ \,for almost all $\,t\in [a,b]$,}
\]
whence $[\gamma]\in L^1([a,b],E)$
is uniquely determined by~$\eta$ (see, for instance, \cite[Lemma~1.28]{MeR}).
We say that a function $\eta\colon I\to E$ on an interval $I\sub \R$
is \emph{absolutely continuous}
if $\eta|_{[a,b]}$ is so for all real numbers $a<b$ such that $[a,b]\sub I$.
If $W\sub \R\times E$ is a subset and $f\colon
W\to E$ a function, then a function
$\gamma\colon I\to E$ on a non-degenerate interval $I\sub\R$
is called a \emph{Carath\'{e}odory solution}
to the differential equation
\[
y'(t)=f(t,y(t))
\]
if $(t,\gamma(t))\in W$ for all $t\in I$,
the function $\gamma$ is absolutely continuous,
and
\begin{equation}\label{cara1}
\gamma'(t)=f(t,\gamma(t))\quad\mbox{for almost all $\,t\in I$.}
\end{equation}
If $(t_0,y_0)\in W$ is given and, moreover,
\begin{equation}\label{rest}
t_0\in I\quad\mbox{and}\quad
\gamma(t_0)=y_0,
\end{equation}
then $\gamma$ is called a Carath\'{e}odory solution
to the \emph{initial value problem}
\begin{equation}\label{the-ivp}
y'(t)=f(t,y(t)),\quad y(t_0)=y_0.
\end{equation}
If $\gamma\colon I \to E$ is absolutely
continuous, $(t,\gamma(t))\in U$ for all $t\in I$
and $t_0\in I $, then $\gamma$ is a Carath\'{e}odory
solution to~(\ref{the-ivp})
if and only if
\begin{equation}\label{cara2}
\gamma(t)=y_0+\int_{t_0}^tf(s,\gamma(s))\, ds\quad\mbox{for all $\,t\in I$.}
\end{equation}
For the special case of
Carath\'{e}odory solutions to differential equations
in Banach spaces, see also~\cite[Chapter~30]{Sch}.\\[2.3mm]
Since $C^1$-functions operate on
absolutely continuous functions,
we can speak about absolutely continuous functions
$\gamma\colon I \to M$ to a $C^1$-manifold~$M$
modelled on a Fr\'{e}chet space
(a continuous function which is absolutely continuous
in local charts). Likewise,
we can also speak about
Carath\'{e}odory solutions
to differential equations on~$M$
(see Sections~\ref{sec-prels} and \ref{sec-now-mfd} for details).\\[2.3mm]
Let $M$ be a $C^1$-manifold
modelled on a locally convex space, $J\sub\R$
be a non-degenerate interval and
\[
f\colon J\times M\to TM
\]
be a map such that $f(t,y)\in T_yM$ for all
$(t,y)\in J\times M$ (a \emph{time-dependent vector field}).
Or, more generally, consider
a function $f\colon W\to TM$
on an open subset $W\sub J\times M$
such that $f(t,y)\in T_yM$ for all $(t,y)\in W$.
It is known that the differential equation
\[
\dot{y}(t)\,=\, f(t,y(t))
\]
satisfies local existence and
uniqueness of $C^1$-solutions if it admits
local $C^1$-flows (see \cite[Lemma~2.5.10]{GaN};
cf.\ \cite[Appendix~A]{Eyn} and \cite[Appendix~A.4]{Wal}
for special cases).
Likewise, we have (with terminology as in Section~\ref{sec-now-mfd}):
\begin{prop}\label{thmA}
Let $M$ be a $C^1$-manifold modelled on a Fr\'{e}chet
space, $J\sub \R$ be a non-degenerate interval and
$f\colon W \to TM$
be a map on an open subset $W\sub J\times M$
such that $f(t,y)\in T_yM$ for all
$(t,y)\in W$.
If
\begin{equation}\label{the-ode-mfd}
\dot{y}(t)\,=\, f(t,y(t))
\end{equation}
admits local flows which are pullbacks of $C^1$-maps,
then {\rm(\ref{the-ode-mfd})}
satisfies local existence and
uniqueness of Carath\'{e}odory
solutions.
\end{prop}
Let $G$ be a Lie group modelled on a Fr\'{e}chet
space with neutral element~$e$
and Lie algebra $\cg=T_eG$.
If $\gamma\colon [a,b]\to \cg$
is an $\cL^1$-map, then the initial value problem
\[
\dot{y}(t)=y(t).\gamma(t),\quad y(a)=e
\]
has at most one Carath\'{e}odory solution
$\eta\colon [a,b]\to G$;
as above, $\Evol([\gamma]):=\Evol(\gamma):=\eta\in AC([a,b],G)$
is called the \emph{evolution} of $\gamma$ (cf.\ \cite{MeR}).
Write
\[
\evol([\gamma]):=\evol(\gamma):=\Evol(\gamma)(b)\in G.
\]
If each $\gamma\in\cL^1([0,1],\cg)$ has an evolution
and
\[
\Evol\colon L^1([0,1],\cg)\to C([0,1],G)
\]
is smooth as a map to the Lie group
$C([0,1],G)$ of continuous $G$-valued maps
on $[0,1]$ (or, equivalently, to the Lie group
$AC([0,1],G)$ of absolutely continuous $G$-valued maps),
then $G$ is called \emph{$L^1$-regular} (see \cite{MeR}).\\[2.3mm]
Every Lie group~$G$ modelled on a Banach space is $L^1$-regular,
as well as $C^\ell(M,G)$ for each compact smooth manifold~$M$
and $\ell\in\N_0\cup\{\infty\}$ (as in~\cite{GCX})
and the Lie group
$\Diff(M)$ of all smooth diffeomorphisms of~$M$ (as in~\cite{Mic},
\cite{Ham} and \cite{Mil}),
see
Theorem~C, Proposition~7.11, and Theorem~C in \cite{MeR},
respectively.\\[2.3mm]
More generally, given $p\in [1,\infty[$ and a Fr\'{e}chet space~$E$,
we can define $\cL^p$-functions $\gamma\colon [a,b]\to E$
and a corresponding Fr\'{e}chet space $L^p([a,b],E)$,
replacing the requirement
(\ref{theL1}) with
\[
\|\gamma\|_{\cL^p,q}:=\sqrt[p]{\int_a^bq(\gamma(t))^p\,dt}<\infty
\]
for each continuous seminorm $q$ on~$E$;
if $p=\infty$, we require that the essential suprema
\[
\|\gamma\|_{\cL^\infty,q}:=\|q\circ\gamma\|_{\cL^\infty}
\]
be finite. In either case, we write $[\gamma]$
for the equivalence class in $L^p([a,b],E)$
and $\|[\gamma]\|_{L^p,q}:=\|\gamma\|_{\cL^p,q}$.
The inclusion $\cL^p([a,b],E)\sub \cL^1([a,b],E)$
induces an injective linear map $L^p([a,b],E)\to L^1([a,b],E)$,
which we use to identify $L^p([a,b],E)$
with a vector subspace of $L^1([a,b],E)$.
Unless the contrary is stated, $L^p([a,b],E)$
shall be endowed with the $L^p$-topology
(the locally convex vector topology given by the
seminorms $\|\cdot\|_{L^p,q}$).\\[2.3mm]
If $p$ is as before
and $G$ is a Lie group modelled on a Fr\'{e}chet space
such that $\Evol(\gamma)$ exists for each $\gamma\in \cL^p([0,1],\cg)$
and $\Evol\colon L^p([0,1],\cg)\to C([0,1],G)$
is smooth, then $G$ is called \emph{$L^p$-regular} (see \cite[\S5]{MeR}).\\[2.3mm]
We obtain the following existence and uniqueness
result for differential equations on $G$-manifolds
given by time-dependent fundamental vector fields.
\begin{thm}\label{thmB}
Let $G$ be a Lie group modelled on a Fr\'{e}chet
space, $M$ be a smooth manifold modelled
on a Fr\'{e}chet space and $\sigma\colon M\times G\to M$
be a smooth map which is a right $G$-action.
If $\gamma\in\cL^1([a,b],\cg)$ has an evolution
$\eta\in AC([a,b],G)$,
then the differential equation
\begin{equation}\label{nowode}
\dot{y}(t)=\gamma(t)_\sharp(y(t))
\end{equation}
on~$M$ satisfies local existence and uniqueness
of Carath\'{e}odory solutions.
The mapping
\begin{equation}\label{flo-form}
\Fl\colon [a,b]\times[a,b]\times M\to M,\quad
(t,t_0,y_0)\mto \sigma(y_0,\eta(t_0)^{-1}\eta(t))
\end{equation}
is the maximal flow of {\rm(\ref{nowode})}.
Moreover, {\rm(\ref{nowode})}
admits local flows which are pullbacks of
$C^\infty$-maps.
\end{thm}
In fact, we shall see that $\Fl$ is such a pullback.
\begin{rem}
Consider a smooth right $G$-action $M\times G\to M$
as before, and $T>0$. As a special case of Theorem~\ref{thmB},
we know that if $\Evol(\gamma)$ exists for an $\cL^1$-map
$\gamma\colon [0,T]\to\cg$, then for $x_0\in M$ the
initial value problem
\[
\dot{y}(t)=\gamma(t)_\sharp(y(t)),\quad y(0)=x_0
\]
has a unique solution on $[0,T]$,
given by $t\mto x_0.\Evol(\gamma)(t)$. The endpoint of this integral
curve, for $t=T$, is
\[
x_0.\evol(\gamma).
\]
\end{rem}
We shall use the following elementary concepts.
\\[2.3mm]
Let $X$ be a set. A function
$\gamma\colon [a,b]\to X$
is called a \emph{staircase function}
if there exist $n\in\N$ and real numbers
\[
a=t_0<t_1<\cdots< t_n=b
\]
such that $\gamma|_{]t_{j-1},t_j[}$ is constant
for all $j\in\{1,\ldots, n\}$.
If $X$ is a topological space,
then a function
$\gamma\colon [a,b]\to X$
is called \emph{piecewise continuous}
if there exist $n\in\N$ and real numbers
\[
a=t_0<t_1<\cdots< t_n=b
\]
such that $\gamma|_{]t_{j-1},t_j[}$
has a continuous extension $[t_{j-1},t_j]\to X$
for all $j\in\{1,\ldots, n\}$.\\[2.3mm]
We write $\im(f):=f(X)$ for the image of a function
$f\colon X\to Y$.\\[2.3mm]
See \cite{AaS}, \cite{HHL}, \cite{Jur},
\cite{JaS}, \cite{Sac}, and the references
therein for geometric control theory
in finite dimensions,
and \cite{Son} for general
aspects of control theory.
Our results concerning control theory subsume the following.
\begin{thm}\label{thmC}
Let $G$ be a Lie group modelled
on a Fr\'{e}chet space, with Lie algebra $\cg$.
Let $S\sub\cg$ be a non-empty subset,
$M$ be a smooth manifold
modelled on a Fr\'{e}chet space,
and $\sigma\colon M\times G\to M$, $(x,g)\mto x.g$ be a smooth map
which is a right $G$-action. Let $x_0, y_0\in M$,
$T\in\; ]0,\infty[$, and $U\sub M$ be an open neighbourhood
of~$y_0$. Let $p\in [1,\infty[$
or $p=\infty$. If $G$ is $L^1$-regular,
then the following
conditions are equivalent.
\begin{itemize}
\item[\rm(a)]
There exists $\gamma\in \cL^1([0,T],\cg)$
with $\im(\gamma)\sub S$ such that $x_0.\evol(\gamma)\in U$.
\item[\rm(b)]
There exists $\gamma\in \cL^p([0,T],\cg)$
with $\im(\gamma)\sub S$ such that $x_0.\evol(\gamma)\in U$.
\item[\rm(c)]
There exists a piecewise continuous function
$\gamma\colon [0,T]\to\cg$
with $\im(\gamma)\sub S$ such that $x_0.\evol(\gamma)\in U$.
\item[\rm(d)]
There exists a staircase function $\gamma\colon [0,T]\to\cg$
with $\im(\gamma)\sub S$ such that $x_0.\evol(\gamma)\in U$.
\end{itemize}
If $S$ is convex, then the following condition~{\rm(e)} is equivalent
to {\rm(d):}
\begin{itemize}
\item[\rm(e)]
There exists a continuous function $\gamma\colon [0,T]\to \cg$
with $\im(\gamma)\sub S$ such that $x_0.\evol(\gamma)\in U$.
\end{itemize}
If $S$ is convex and the convex hull of the set $\ex(S)$
of extreme points is dense in $S$ $($e.g.,
if $S$ is compact and convex$)$,
then the following condition~{\rm(f)} is equivalent to {\rm(d):}
\begin{itemize}
\item[\rm(f)]
There exists a staircase function $\gamma\colon [0,T]\to\cg$
with $\im(\gamma)\sub \ex(S)$ such that $x_0.\evol(\gamma)\in U$.
\end{itemize}
\end{thm}
\begin{rem}
Note that (f) is an instance of a bang-bang principle:
If we can enter~$U$ in time~$T$ using controls
in~$S$, then also using piecewise constant controls in the set
$\ex(S)$ of extreme points.
To prove~(f), we shall use
the Trotter formula,
which is valid for all $L^1$-regular Lie groups, by~\cite[Theorem~I]{MeR}
(cf.\ \cite{Tro} for further generalizations).
\end{rem}
\begin{rem}
With regard to~(f), we recall that the convex hull $\conv(\ex(S))$ is dense in~$S$
for each weakly
compact convex subset $S$ of a locally convex space~$E$,
i.e., a convex subset which is compact in the weak
topology $\cO_w$ on~$E$ (which is initial with respect to
the set~$E'$ of continuous linear functionals on~$E$).\footnote{By
the Krein-Milman Theorem, $\conv(\ex(S))$ has closure~$S$
in $(E,\cO_w)$. The closure $C$ of $\conv(\ex(S))$ in~$E$
satisfies $C\sub S$, as $\cO_w$ is coarser than the given topology on~$E$.
Being closed and convex, $C$ is an intersection of closed half-spaces $H\sub E$,
by the Hahn-Banach Separation Theorem.
As each $H$ is weakly closed, $C$ is closed
in $(E,\cO_w)$ and so~$C=S$.}
For example, $\conv(\ex(S))$ is dense in~$S$
for the closed unit ball~$S$ in any reflexive Banach space~$E$
(e.g., in a Hilbert space); likewise for every closed, convex, bounded subset
$S\sub E$.
\end{rem}
\begin{rem}\label{evolstaircase}
Recall that every $L^1$-regular Lie group~$G$
has an exponential function $\exp_G\colon \cg\to G$.
If $v\in\cg$ is given and real numbers $\alpha<\beta$,
then the constant function $\gamma\colon [\alpha,\beta]\to\cg$, $t\mto v$
satisfies
\[
\evol(\gamma)=\exp_G((\beta-\alpha)v).
\]
If $T>0$ and $\gamma\colon [0,T]\to \cg$ is a staircase
function, let $0=t_0<\cdots<t_n=T$ such that
$\gamma$ has a constant value $v_j$ on $]t_{j-1},t_j[$
for $j\in\{1,\ldots, n\}$. Then
\[
\evol(\gamma)=\exp_G((t_1-t_0)v_1)\exp_G((t_2-t_1)v_2)\cdots\exp_G((t_n-t_{n-1})v_n).
\]
\end{rem}
\begin{rem}
We can interpret Theorem~\ref{thmC}
as a result concerning the closures of reachable
sets. E.g., given a subset $S\sub\cg$ and $x_0\in M$, let
\[
\Reach_S(x_0)
\]
be the set of all $y_0\in M$ such that $y_0=x_0$ or
$y_0=x_0.\evol(\gamma)$ for some
$T>0$ and $\gamma\in \cL^1([0,T],\cg)$ with $\im(\gamma)\sub S$.
Using Remark~\ref{evolstaircase},
we deduce from Theorem~\ref{thmC}\,(d) that
\[
\wb{\Reach_S(x_0)}=\wb{x_0.\langle\exp_G([0,\infty[\,S)\rangle_+},
\]
where $\langle Y \rangle_+$ denotes the subsemigroup
of~$G$ generated by a subset $Y \sub G$.
\end{rem}
For some conclusions, weaker regularity properties
(like $L^\infty$-regularity) are sufficient.
To enable these variants, we discuss
continuity of the evolution map
$\Evol\colon L^p([0,1],\cg)\to C([0,1],G)$
with respect to the $L^1$-topology on its domain,
given by the seminorms
\[
L^p([0,1],\cg)\to[0,\infty[,\quad
[\gamma]\mto\int_0^1q(\gamma(t))\,dt
\]
for continuous seminorms~$q$ on~$\cg$.
\begin{thm}\label{thmD}
Let $p\in [1,\infty[$ or $p=\infty$.
Then
\[
\Evol\colon L^p([0,1],\cg)\to C([0,1],G)
\]
is continuous with respect to the $L^1$-topology
on $L^p([0,1],\cg)$, for each $L^p$-regular Lie group~$G$
modelled on a Fr\'{e}chet space.
\end{thm}
It was already shown in \cite[Lemma~14.9]{SEM}
that $\evol\colon C([0,1],\cg)\to G$
is continuous with respect to the $L^1$-topology
for each $C^0$-regular, locally
$\mu$-convex Lie group~$G$,
and we can adapt the proof
(the concept of local $\mu$-convexity,
which goes back to~\cite{SEM},
is recalled in Definition~\ref{locmu}).
In the meantime, work by
Hanusch showed that every
$C^0$-regular Lie group is locally $\mu$-convex
(cf.\ Theorem~1 in \cite[\S5]{Han}).
Since $L^p$-regularity implies $C^0$-regularity
(see \cite[Corollary 5.21]{MeR}),
we can exploit that~$G$ in Theorem~\ref{thmD}
is locally $\mu$-convex.\\[2.3mm]
We mention that more can be shown:
The evolution map in Theorem~\ref{thmD}
is~$C^\infty$ with respect to
the $L^1$-topology (see \cite[Remark 4.3]{Ana}).\\[2.3mm]
Using Theorem~\ref{thmD}
as a tool, we can generalize Theorem~\ref{thmC}
as follows:
\begin{thm}\label{thmE}
Instead of requiring $L^1$-regularity,
let $G$ be a Lie group modelled
on a Fr\'{e}chet space
such that $G$ is $L^q$-regular
for some $q\in \;]1,\infty[$ or $q=\infty$.
Then all conclusions of Theorem~{\rm\ref{thmC}}
remain valid if we assume $p\geq q$
and replace~{\rm(a)} with
\begin{itemize}
\item[\rm(a)$'$]
There exists $\gamma\in \cL^q([0,T],\cg)$
with $\im(\gamma)\sub S$ such that $x_0.\evol(\gamma)\in U$.
\end{itemize}
\end{thm}
A variant of Theorem~\ref{thmC} is also available
if a Lie group~$G$ is only assumed to be $C^0$-regular
(see Theorem~\ref{thmF}).
\begin{rem}
Using Lusin measurability instead of Borel
measurability, it is possible to define
$\cL^p$-maps to sequentially
complete locally convex spaces,
corresponding absolutely continuous maps,
and $L^p$-regularity
(see \cite{Nik}).
Using these, we
find that Proposition~\ref{thmA}
remains valid if $M$ is a $C^2$-manifold
modelled on a sequentially complete
locally convex space and (\ref{the-ode-mfd})
admits local flows which are pullbacks of~$C^2$-maps
(see Remark~\ref{thmAgen}\,(a));
Theorems~\ref{thmB}, \ref{thmC},
and~\ref{thmE}
remain valid if $G$ is a Lie group
modelled on a sequentially complete
locally convex space and~$M$ a smooth
manifold modelled
on a sequentially complete locally convex space, with $L^1$-regularity (and $L^q$-regularity)
as in~\cite{Nik}
(see Section~\ref{sec-beyond},
notably Theorem~\ref{thmBgen},
Remark~\ref{givesC}, and Remark~\ref{givesDE}).
Moreover, Theorem~\ref{thmD} remains
valid for Lie groups modelled
on sequentially complete locally convex spaces
and $L^p$-regularity as in \cite{Nik}
(see Remark~\ref{givesDE}).
Generalizations to $\cE$-regular Lie groups
modelled on sequentially complete (FEP)-spaces
or integral complete locally convex spaces
(as in \cite{MeR}) are also possible,
see Remark~\ref{furth-var} (for the terminology,
cf.\ also Section~\ref{sec-han}).
\end{rem}
Absolutely continuous functions
$\eta\colon [a,b]\to E$
as in (\ref{gammaprimi}) to a sequentially
complete locally convex space
are more difficult to treat than those
to Fr\'{e}chet spaces,
as $\eta'(t)$ may
not exist (and recover $\gamma(t)$)
almost everywhere in this case.
Notably, Carath\'{e}odory
solutions need to be understood in the
sense of~(\ref{cara2}),
while~(\ref{cara1}) might not hold
almost everywhere.
At least, for each
continuous linear map $q\colon E\to F$
to a Banach space~$F$, we still have that
$(q\circ \eta)'(t)$ exists and equals
$q(\gamma(t))$ for almost all~$t$.
This will be good enough for the proofs,
exploiting that $q\circ f$ for an $E$-valued
$C^2$-function $f$ locally factors over a $C^1$-map on an open
subset of a Banach space (see Appendix~\ref{appA}).
\begin{rem}
Examples of $L^1$-regular Lie groups
modelled on sequentially complete
locally convex spaces are direct
limits $G=\bigcup_{n\in\N} G_n$
of finite-dimensional Lie groups
$G_1\sub G_2\sub\cdots$ (as in \cite{DIR}),
the Lie group $\Diff^\omega(M)$
of real-analytic diffeomorphisms
of a compact real-analytic manifold~$M$
(as in \cite{KaM})
and the Lie group $\Diff_c(M)$
of compactly supported smooth diffeomorphisms
of a finite-dimensional paracompact smooth manifold~$M$
(cf.\ \cite{Mic}); $L^1$-regularity was
established in \cite[Theorem~E]{MeR},
\cite[Theorem~1.2]{Ana}, and \cite[Theorem~B]{MeR},
respectively.
See Theorems 1.1 and 1.7 in \cite{Ana}
for applications to
the natural action
of $\Diff^\omega(M)$ (resp., $\Diff_c(M)$)
on~$M$.
\end{rem}
\section{Preliminaries and notation}\label{sec-prels}
In the following, $\N:=\{1,2,\ldots\}$
and $\N_0:=\N\cup\{0\}$.
All topological vector spaces
are assumed Hausdorff,
with the exception of the spaces
$\cL^p([a,b],E)$ for a Fr\'{e}chet space~$E$,
real numbers $a<b$ and $p\in[1,\infty]$,
which are endowed with the
vector topology determined by the seminorms
$\|\cdot\|_{\cL^p,q}$ for continuous
seminorms~$q$ on~$E$.
The topology induced by $\cL^1([a,b],E)$
on $\cL^p([a,b],E)$ shall be referred to as the
\emph{$\cL^1$-topology} thereon.
We shall use ``locally convex space''
as an abbreviation of ``locally convex topological vector space.''
If $M$ is a subset of a real vector space~$E$,
we write $\conv(M)$ for its convex hull.
If $p\colon E\to [0,\infty[$
is a seminorm, we write $B^p_\ve(x):=\{y\in E\colon p(y-x)<\ve\}$
for the open ball of radius $\ve>0$
around $x\in E$. If $(E,\|\cdot\|)$
is a normed space and the norm is understood, we also write
$B^E_\ve(x)$ in place of $B^{\|\cdot\|}_\ve(x)$.
If $E$ and $F$ are locally convex spaces
and $U\sub E$ is an open subset,
we say that a continuous map $f\colon U\to F$
is $C^1$ if the directional derivative
\[
df(x,y):=\frac{d}{dt}\Big|_{t=0}f(x+ty)
\]
exists in~$F$ for all $(x,y)\in U\times E$,
and the map $df\colon U\times E\to F$
is continuous.
Recursively, given $k\geq 2$ we say that
$f$ is $C^k$ if $f$ is $C^1$ and $df$ is $C^{k-1}$.
If $f$ is $C^k$ for all $k\in\N$,
then $f$ is called $C^\infty$ or \emph{smooth}.
This approach to calculus in locally
convex spaces, which goes back to \cite{Bas},
is known as Keller's $C^k_c$-theory~\cite{Kel}.
We refer to \cite{Res}, \cite{GaN},
\cite{Ham}, \cite{Mil}, and \cite{Nee} for introductions
to this approach to calculus, cf.\ also \cite{BGN}.
For the corresponding concepts of manifolds
and Lie groups modelled on a locally convex space,
see~\cite{Res}, \cite{GaN}, and~\cite{Nee}.
As usual, a Lie group (resp., manifold)
modelled on a Fr\'{e}chet space
shall be called a \emph{Fr\'{e}chet-Lie group}
(resp., a \emph{Fr\'{e}chet manifold}).
If $M$ is a $C^1$-manifold modelled on a locally
convex space, we let $TM$ be its tangent bundle
and write $T_xM$ for the tangent space at $x\in M$.
If $V$ is an open subset of a locally convex space~$E$,
we identify $TV$ with $V\times E$, as usual.
If $f\colon M\to N$ is a $C^1$-map between $C^1$-manifolds, we write
$Tf\colon TM\to TN$ for its tangent map.
In the case of a $C^1$-map $f\colon M\to V\sub E$,
we write $df$ for the second component of the tangent map
\[
Tf\colon TM\to TV=V\times E.
\]
Many preliminaries were already described in the introduction,
and need not be repeated.
%here.
For more background
concerning vector-valued $\cL^p$-functions,
vector-valued absolutely continuous functions,
and $L^p$-regularity, see~\cite{MeR}.
\begin{numba}\label{def-abs-mfd}
If $M$ is a $C^1$-manifold modelled on a Fr\'{e}chet space
and $I\sub\R$ an interval,
we say that a function $\eta\colon I\to M$
is \emph{absolutely continuous}
if $\eta$ is continuous
and, for each $t_0\in I$,
there exist a chart $\phi\colon U_\phi\to V_\phi$ of~$M$
and real numbers $\alpha<\beta$ with $[\alpha,\beta]\sub I$
such that $t_0$ is in the interior of~$[\alpha,\beta]$
relative~$I$, $\eta([\alpha,\beta])\sub U_\phi$,
and $\phi\circ \eta|_{[\alpha,\beta]}$
is absolutely continuous.
Equivalently, $\phi\circ \eta|_J$
is absolutely continuous for each chart $\phi\colon U_\phi\to V_\phi$
of~$M$ and each interval $J\sub I$ such that $\eta(J) \sub U_\phi$
(cf.\ Definition 3.20, Lemma~3.21, Lemma~3.18(a)
and 3.15 in \cite{MeR}).
\end{numba}
Let $I\sub \R$ be a non-degenerate interval and $E$ a locally convex space.
As usual, we say that a map $\eta\colon I\to E$ is \emph{differentiable}
at $t_0\in I$ if the limit
\[
\eta'(t_0):=\lim_{t\to t_0}\frac{\eta(t)-\eta(t_0)}{t-t_0}
\]
(with $t\not=t_0$) exists in~$E$.
We shall use a well-known fact (cf.\ \cite[Lemma~1.57]{MeR}):
\begin{numba}\label{pwchain}
Let $E$ and $F$ be locally convex spaces, $U\sub E$ be open and
$f \colon U \to F$ be a $C^1$-map. If
$I \sub \R$ is a non-degenerate interval,
$\eta\colon  I \to E$ a function with $\eta(I)\sub U$ and
$t_0\in I$ such that $\eta'(t_0)$ exists, then also $(f\circ\eta)'(t_0)$
exists and $(f\circ\eta)'(t_0)=df(\eta(t_0),\eta'(t_0))$.
\end{numba}
We deduce from~\ref{pwchain}:
\begin{numba}\label{C1actsAC}
Let $E$ and $F$ be Fr\'{e}chet spaces, $U\sub E$ be open and
$f \colon U \to F$ be a $C^1$-map. If
$I \sub \R$ is a non-degenerate interval
and $\eta\colon  I \to E$ an absolutely continuous function
such that $\eta(I)\sub U$,
then $f\circ\eta\colon I\to F$ is absolutely continuous
(see \cite[Lemma~3.18\,(a)]{MeR}).
For each $t_0\in I$ such that $\eta'(t_0)$ exists
$($which is the case for $\lambda_1$-almost all $t_0\in I)$,
we have that $(f\circ\eta)'(t_0)$
exists and $(f\circ\eta)'(t_0)=df(\eta(t_0),\eta'(t_0))$.
\end{numba}
\begin{numba}\label{dotty}
Let $M$ be a $C^1$-manifold modelled
on a locally convex space~$E$. Let $I\sub\R$ be a non-degenerate
interval, $\eta\colon I\to M$
be a continuous map and $t_0\in I$.
We say that \emph{$\eta$ is differentiable
at $t_0$} if $\phi\circ \eta\colon \eta^{-1}(U_\phi)\to V_\phi$
is differentiable at $t_0$ for some chart
$\phi\colon U_\phi\to V_\phi\sub E$ of~$M$
such that $\eta(t_0)\in U_\phi$.
By~\ref{pwchain},
the latter then holds for any such chart,
and the tangent vector
\[
\dot{\eta}(t_0):=T\phi^{-1}((\phi\circ\eta)(t_0),(\phi\circ\eta)'(t_0))
\in T_{\eta(t_0)}M
\]
is well defined, independent of the choice of~$\phi$.
If no confusion is likely,
we also write $\frac{d\eta}{dt}(t_0):=\dot{\eta}(t_0)$
(e.g., in Definition~\ref{def-loc-flo-mfd}\,(c)).
\end{numba}
\begin{numba}\label{C1onpw}
Let $f\colon M\to N$ be a $C^1$-map between
$C^1$-manifolds modelled
on locally convex spaces. If $I\sub \R$
is a non-degenerate interval, $t_0\in I$
and $\eta\colon I\to M$ a continuous map which is differentiable
al~$t_0$, then also $f\circ \eta\colon I\to N$ is differentiable at~$t_0$
and
\begin{equation}\label{tangpw}
(f\circ\eta)^{\!\textstyle \cdot}(t_0)=Tf(\dot{\eta}(t_0)).
\end{equation}
[Let $\psi\colon U_\psi\to V_\psi$
be a chart of $N$ with $f(\eta(t_0))\in U_\psi$
and $\phi\colon U_\phi\to V_\phi$
be a chart of $M$ with $\eta(t_0)\in U_\phi$
and $f(U_\phi)\sub U_\psi$. Let $J\sub I$ be
a subinterval which is a neighbourhood of $t_0$ in~$I$,
such that $\eta(J)\sub U_\phi$. Then
\[
\psi\circ f\circ\eta|_J=(\psi\circ f\circ \phi^{-1})\circ(\phi\circ\eta|_J)
\]
is differentiable at~$t_0$, by \ref{pwchain}, whence so is $f\circ\eta$.
Moreover, using~\ref{pwchain}, we obtain
$(\psi(f(\eta(t_0))),(\psi\circ f\circ\eta)'(t_0))=
T(\psi\circ f\circ \phi^{-1})(\phi(\eta(t_0)),(\phi\circ\eta)'(t_0))
=T\psi Tf\dot{\eta}(t_0)$. Applying $T\psi^{-1}$ to both sides,
(\ref{tangpw}) follows.\,]
\end{numba}
\begin{numba}\label{pardiff}
If $E_1$, $E_2$, and $F$ are locally convex spaces,
$U_1\sub E_1$ and $U_2\sub E_2$ open subsets and $f\colon
U_1\times U_2\to F$ a $C^1$-map,
then
\[
df((x_1,x_2),(y_1,y_2))=d_1f(x_1,x_2;y_1)+d_2f(x_1,x_2;y_2)
\]
for all $(x_1,x_2)\in U_1\times U_2$ and $(y_1,y_2)\in E_1\times E_2$,
where
\[
d_1f(x_1,x_2,y_1) :=d(f(\cdot,x_2))(x_1,y_1)
\]
and $d_2f(x_1,x_2,y_2):=d(f(x_1,\cdot))(x_2,y_2)$,
see \cite[Proposition~1.2.8]{GaN}.
Likewise,
\begin{eqnarray*}
\lefteqn{df((x_1,x_2,x_3),(y_1,y_2,y_3))=}\qquad\qquad\\
& & d_1f(x_1,x_2,x_3;y_1)+d_2f(x_1,x_2,x_3;y_2)
+d_3f(x_1,x_2,x_3;y_3)
\end{eqnarray*}
for $C^1$-maps $f\colon U_1\times U_2\times U_3\to F$,
in terms of partial differentials.
\end{numba}
If $X$ is a topological space,
we write $\cB(X)$ for the $\sigma$-algebra
of Borel sets (generated by the set of open
subsets of~$X$).
A map between topological spaces is called
\emph{Borel measurable}
if it is measurable with respect to the
$\sigma$-algebras of Borel sets on
domain and range.
As usual, we say that a topological space is
\emph{separable}
if it has a dense, countable subset.
The following fact is useful.
\begin{la}\label{compo-prod}
Let $X$, $X_1$, $X_2$, and $Y$ be topological
spaces and $f\colon X_1\times X_2\to Y$
be a continuous map.
Let $\gamma\colon X\to X_1$ and $\eta\colon X\to X_2$
be Borel measurable mappings.
If $X=\bigcup_{n\in\N}A_n$
with Borel sets $A_n$ such that $\eta(A_n)$
is separable and metrizable
in the topology induced by~$X_2$,
then
\[
f\circ(\gamma,\eta)\colon X\to Y,\quad x\mto f(\gamma(x),\eta(x))
\]
is Borel measurable.
\end{la}
\begin{proof}
It suffices to show that $f\circ(\gamma,\eta)|_{A_n}$
is measurable on~$A_n$, endowed with the trace
$\cB(X)|_{A_n}=\cB(A_n)$, for each $n\in\N$.
Since $\eta(A_n)$ is metrizable and separable,
\[
\cB(X_1\times X_2)|_{X_1\times\eta(A_n)}=
\cB(X_1\times\eta(A_n))=\cB(X_1)\otimes\cB(\eta(A_n))
\]
is the product $\sigma$-algebra (see, e.g., \cite[Lemma~2.7]{Mea}).
Hence $(\gamma,\eta)|_{A_n}\colon X\to X_1\times X_2$
is measurable as a map to $X_1\times \eta(A_n)$
with the trace of $\cB(X_1\times X_2)$ and hence
Borel measurable to $X_1\times X_2$.
Since $f$ is continuous and thus Borel measurable,
also the composition $f\circ (\gamma,\eta)|_{A_n}$ is
Borel measurable.
\end{proof}
\begin{numba}\label{repara}
Let $G$ be a Lie group with Lie algebra $\cg:=T_eG$,
modelled on a locally convex space.
Let $\alpha<\beta$ and $a<b$ be real numbers
and $\phi\colon [\alpha,\beta]\to[a,b]$
be the restriction of the unique affine-linear map
$\R\to\R$ taking $\alpha$ to $a$ and $\beta$ to~$b$.
Thus $\phi'$ is the constant function
whose value is the slope $m:=(b-a)/(\beta-\alpha)$.
If $\gamma\colon [a,b]\to \cg$ is a continuous
function admitting an evolution $\eta:=\Evol(\gamma)\colon [a,b]\to G$,
then $\zeta\colon [\alpha,\beta]\to \cg$, $s\mto m\,\gamma(\phi(s))$
has $\eta\circ \phi$ as its evolution, i.e.,
\[
\Evol(\zeta)=\Evol(\gamma)\circ\phi.
\]
In fact, $(\eta\circ\phi)(\alpha)=\eta(a)=e$ holds
and $(\eta\circ\phi)^{\!\textstyle\cdot}(s)=\dot{\eta}(\phi(s))\phi'(s)
=\phi'(s)\eta(\phi(s)).\gamma(\phi(s))=
(\eta\circ \phi)(s).(m\,\gamma(\phi(s)))=(\eta\circ\phi)(s).\zeta(s)$
for all $s\in [\alpha,\beta]$, by the Chain Rule.\\[2.3mm]
Likewise if $G$ is a Fr\'{e}chet-Lie group,
$\gamma \in\cL^1([a,b],\cg)$
and $\eta$ its evolution in the sense
of Carath\'{e}odory solutions (cf.\ \cite{MeR}).
\end{numba}
\section{Initial value problems in Fr\'{e}chet spaces}\label{sec-ex-uni}
We discuss local existence and uniqueness
for Carath\'{e}odory solutions to initial
value problems in Fr\'{e}chet spaces.
The treatment emulates the earlier discussion
of existence and uniqueness of $C^1$-solutions
in~\cite[\S2.4]{GaN}.
\begin{defn}\label{defn-loc-uni}
Let $E$ be a Fr\'{e}chet space
and $f\colon W\to E$ be a function on a subset
$W\sub \R\times E$.
We say that the differential equation
\begin{equation}\label{nowtheode}
y'(t)\,=\, f(t,y(t))
\end{equation}
satisfies \emph{local uniqueness of Carath\'{e}odory
solutions} if the following holds:
For all Carath\'{e}odory solutions
$\gamma_1\colon I_1\to E$ and $\gamma_2\colon I_2\to E$
of (\ref{nowtheode}) and $t_0\in I_1\cap I_2$
such that $\gamma_1(t_0)=\gamma_2(t_0)$,
there exists an interval $K\sub\R$
which is an open neighbourhood of~$t_0$
in $I_1\cap I_2$ such that
\[
\gamma_1|_K=\gamma_2|_K.
\]
\end{defn}
\begin{la}\label{la-loc-glob}
In a Fr\'{e}chet space~$E$,
consider a differential equation {\rm(\ref{nowtheode})}
which satisfies local uniqueness of Carath\'{e}odory
solutions. Assume that\linebreak
$\gamma_1\colon I_1\to E$ and $\gamma_2\colon I_2\to E$
are Carath\'{e}odory solutions to {\rm(\ref{nowtheode})}
such that $\gamma_1(t_0)=\gamma_2(t_0)$ for some
$t_0\in I_1\cap I_2$. Then
\[
\gamma_1|_{I_1\cap I_2}=\gamma_2|_{I_1\cap I_2}.
\]
\end{la}
\begin{proof}
(Compare \cite[Lemma~2.4.6]{GaN} for $C^1$-solutions).
The set $A:=\{t\in I_1\cap I_2\colon \gamma_1(t)=\gamma_2(t)\}$
is closed in $I_1\cap I_2$ since~$E$ is Hausdorff and
the functions $\gamma_1$ and $\gamma_2$ are continuous.
Since~(\ref{nowtheode}) satisfies local uniqueness
of Carath\'{e}odory solutions, the set~$A$ is also open in $I_1\cap I_2$.
By hypothesis, $A\not=\emptyset$. Since $I_1\cap I_2$
is an interval and hence connected, it follows that $A=I_1\cap I_2$
and thus $\gamma_1|_{I_1\cap I_2}=\gamma_2|_{I_1\cap I_2}$.
\end{proof}
\begin{defn}\label{defn-loc-ex}
Let $E$ be a Fr\'{e}chet space,
$J\sub\R$ a non-degenerate interval
and $f\colon W\to E$ be a function on a subset
$W\sub J\times E$. We say that the differential equation (\ref{nowtheode})
satisfies \emph{local existence of Carath\'{e}odory
solutions}\footnote{More precisely, we should speak about
local existence of Carath\'{e}odory solutions
\emph{with respect to~$J$}, but $J$ will always be clear from
the context. Likewise in Definition~\ref{def-locex-mfd}.} if for all $(t_0,y_0)\in W$,
there exists a Carath\'{e}odory solution
$\gamma\colon I\to E$ to
the initial value problem
\begin{equation}\label{theiniprob}
y'(t) \,=\, f(t,y(t)),\quad y(t_0)=y_0
\end{equation}
such that $I$ is a
relatively open subinterval of~$J$.
\end{defn}
\begin{defn}\label{defn-loc-flo}
Let $J\sub \R$ be a non-degenerate interval, $E$ be a Fr\'{e}chet space,
$U\sub E$ be a subset
and $f\colon W\to E$ be a function on an open subset $W\sub J\times U$.
Let $k\in\N\cup\{\infty\}$.
We say that the differential equation $y'(t)=f(t,y(t))$
\emph{admits local flows which are pullbacks of $C^k$-maps}
if,
for all $(\wb{t},\wb{y}) \in W$,
there exist a relatively open interval $I\sub J$ with $\wb{t}\in I$,
an open neighborhood~$V$ of $\wb{y}$ in~$U$
with $I\times V\sub W$
and function
\[
\Phi\colon I\times I\times V\to E
\]
with the following properties:
\begin{itemize}
\item[(a)]
For all $(t_0,y_0)\in I\times V$, the function
$I\to E$, $t\mto \Phi_{t,t_0}(y_0):=\Phi(t,t_0,y_0)$ is a Carath\'{e}odory
solution to
the initial value problem (\ref{theiniprob});
\item[(b)]
There is an open $\wb{y}$-neighbourhood $Y\sub V$
such that $\Phi_{t_1,t_0}(Y)\sub V$ for all $t_0,t_1\in I$ and
\[
\Phi_{t_2,t_1}(\Phi_{t_1,t_0}(y_0))=\Phi_{t_2,t_0}(y_0)
\;\;\mbox{for all $\,t_0,t_1,t_2\in I$ and $y_0\in Y$;}
\]
\item[(c)]
There exist Fr\'{e}chet spaces~$E_1$ and $E_2$,
open subsets $V_1\sub E_1$ and $V_2\sub E_2$,
absolutely continuous functions
$\alpha\colon I\to V_1\sub E_1$ and $\beta\colon I\to V_2\sub E_2$,
and a $C^k$-map $\Psi\colon V_1\times V_2 \times V\to E$
such that\footnote{In particular, $\Phi$ is continuous.}
\[
\Phi(t,t_0,y_0)=\Psi(\alpha(t),\beta(t_0),y_0)\quad\mbox{for all $(t,t_0,y_0)
\in I\times I\times V$.}
\]
Moreover, we require the existence of a Borel set $I_0\sub I$
with $\lambda_1(I\setminus I_0)=0$ such that
$\frac{d}{dt}\Phi_{t,t_0}(y_0)$
exists
and
\[
\frac{d}{dt}\Phi_{t,t_0}(y_0)=f(t,\Phi_{t,t_0}(y_0))
\]
for all $y_0\in V$, $t_0\in I$, and $t \in I_0$.
\end{itemize}
\end{defn}
\begin{rem}\label{flo-then-ex}
If $y'(t)=f(t,y(t))$ (as in Definition~\ref{defn-loc-flo})
admits local flows which are pullbacks
of $C^k$-maps, then $y'(t)=f(t,y(t))$ satisfies
local existence of Carath\'{e}odory solutions.\\[2.3mm]
In fact,
for any $(\wb{t},\wb{y}) \in W$,
the map $I\to E$, $t\mto\Phi(t,\wb{t},\wb{y})$
is a solution to the initial value problem
$y'(t)=f(t,y(t))$, $y(\wb{t})=\wb{y}$
on the relatively open subinterval $I\sub J$
with $\wb{t}\in I$
(with notation
as in Definition~\ref{defn-loc-flo}).
\end{rem}
\begin{la}\label{locflow-uniq}
Let $J\sub \R$ be a non-degenerate interval, $E$ be a Fr\'{e}chet
space,
$U\sub E$ be a subset
and $f\colon W\to E$ be a function on an open subset $W\sub J\times U$.
If the differential equation $y'(t)=f(t,y(t))$
admits local flows which are pullbacks of $C^1$-maps,
then it satisfies local uniqueness
of Carath\'{e}odory solutions.
\end{la}
\begin{proof}
(Compare \cite[Proposition~2.4.20]{GaN} for local $C^1$-flows).
Let $\gamma_j\colon I_j\to E$
be solutions to $y'(t)=f(t,y(t))$
for $j\in \{1,2\}$ and $\wb{t}\in I_1\cap I_2$ such that
$\wb{y}:=\gamma_1(\wb{t})=\gamma_2(\wb{t})$.
To see that $\gamma_1$ and $\gamma_2$ coincide on a neighborhood of~$\wb{t}$
in $I_1\cap I_2$, we may assume
%that
$I_1\cap I_2\not=\{\wb{t}\}$
(excluding a trivial case). Thus $I_1\cap I_2$ is a non-degenerate interval.
For $j\in \{1,2\}$, there exists a Borel set $I_{j,0}\sub I_j$
with $\lambda_1(I_j\setminus I_{j,0})=0$
such that $\gamma_j$ is differentiable at each $t\in I_{j,0}$
and
\[
\gamma_j'(t)\, =\, f(t,\gamma_j(t))\mbox{ \,for all $\, t\in I_{j,0}$.}
\]
Let $I$, $V$, $\Phi$, $Y$, $E_1$, $E_2$, $V_1$, $V_2$,
$\Psi$, $\alpha$, $\beta$, and $I_0$
be as in Definition~\ref{defn-loc-flo}.
After shrinking~$I_0$,
we may assume that, moreover, $\alpha'(t)$ and $\beta'(t)$
exist at each $t\in I_0$.
For $(t,y_0)\in I\times V$,
the partial derivative of~$\Phi$ with respect to the second variable,
\[
(\partial_2\Phi)(t,t_0,y_0):=\frac{\partial\Phi}{\partial t_0}(t,t_0,y_0)
=d_2\Psi(\alpha(t),\beta(t_0),y_0;\beta'(t_0)),
\]
exists for all $t_0\in I$ such that $\beta'(t_0)$
exists, and hence for all $t_0\in I_0$
(where we used~\ref{C1actsAC}
and notation as in \ref{pardiff}).\\[2.3mm]
There exists a relatively open interval $K\sub I_1\cap I_2\cap I$
with $\wb{t}\in K$ such that $\gamma_1(K)\sub Y$,
$\gamma_2(K)\sub Y$ and $\Phi_{t,\wb{t}}(\wb{y})\in Y$ for all $t\in K$.
After shrinking~$K$ if necessary, we can also assume that
\[
\theta_j(t):=\Phi_{\wb{t},t}(\gamma_j(t))\in Y \;\;\mbox{for all
$t\in K$ and $j\in\{1,2\}$.}
\]
Note that $\theta_j$ is absolutely continuous
by~\ref{C1actsAC} as
\[
\theta_j(t)=\Psi(\alpha(\wb{t}),\beta(t),\gamma_j(t));
\]
moreover,
\begin{eqnarray}
\theta_j'(t) &=& d_2\Psi(\alpha(\wb{t}),\beta(t),\gamma_j(t);\beta'(t))
+d_3\Psi(\alpha(\wb{t}),\beta(t),\gamma_j(t);\gamma_j'(t))\notag \\
&=&
\partial_2\Phi(\wb{t},t,\gamma_j(t))+d\Phi_{\wb{t},t}(\gamma_j(t),
\gamma_j'(t))\label{inamin}
\end{eqnarray}
for all $t\in K_0:=K\cap I_0\cap I_{1,0}\cap I_{2,0}$.
It suffices to show that
\[
\gamma_j(t)=\Phi_{t,\wb{t}}(\wb{y})\;\;\mbox{for
all $t\in K$ and $j\in\{1,2\}$.}
\] 
Since $\Phi_{t,\wb{t}}\circ \Phi_{\wb{t},t}|_Y=\id_Y$ for all
$t\in I$ (by (a) and (b) in Definition~\ref{defn-loc-flo}), the map $\Phi_{\wb{t},t}|_Y$ is injective.
Hence $\gamma_1|_K=\gamma_2|_K$
will hold if we can show that both $\theta_1$ and $\theta_2$ coincide
with
\[
\theta\colon K\to E,\quad t\mto \Phi_{\wb{t},t}(\Phi_{t,\wb{t}}(\wb{y}))=\wb{y}.
\]
Since $\theta_j(\wb{t})=\wb{y}=\theta(\wb{t})$ for $j\in\{1,2\}$,
the latter will hold if we can show that
\[
\theta_j'(t)=\theta'(t)=0\;\;\mbox{for all $t\in K_0$}
\]
(and thus for $\lambda_1$-almost all $t\in K$).
Given $z\in Y$, we have
\[
z=\Phi_{\wb{t},t}(\Phi_{t,\wb{t}}(z))=\Psi(\alpha(\wb{t}),\beta(t),\Phi_{t,\wb{t}}(z))
\]
for all $t\in I$
(by (a), (b), and~(c) in Definition~\ref{defn-loc-flo})
and hence, differentiating with respect to~$t$,
\begin{eqnarray}
0 & = & d_2\Psi(\alpha(\wb{t}),\beta(t),\Phi_{t,\wb{t}}(z);\beta'(t))
+d_3\Psi\big(\alpha(\wb{t}),\beta(t),\Phi_{t,\wb{t}}(z),
{\textstyle\frac{d}{dt}}\Phi_{t,\wb{t}}(z)\big)\notag\\
&=&
\partial_2\Phi(\wb{t},t,\Phi_{t,\wb{t}}(z))+d\Phi_{\wb{t},t}(\Phi_{t,\wb{t}}(z),f(t,\Phi_{t,\wb{t}}(z)))\label{tool-flw}
\end{eqnarray}
for all $t\in I_0$
(using that $\frac{d}{dt}\Phi_{t,\wb{t}}(z)=f(t,\Phi_{t,\wb{t}}(z))$
for all $t\in I_0$).
For $t\in K_0$
and $j\in \{1,2\}$,
we can consider $z:=\theta_j(t)\in Y$;
then $\gamma_j(t)=\Phi_{t,\wb{t}}(z)$ .
Since $\gamma_j'(t)=f(t,\gamma_j(t))=f(t,\Phi_{t,\wb{t}}(z))$, using
(\ref{inamin}) we get
\begin{align*}
\theta_j'(t) 
&=\partial_2\Phi(\wb{t},t,\gamma_j(t))+d\Phi_{\wb{t},t}(\gamma_j(t),\gamma_j'(t))\\
&=\partial_2\Phi(\wb{t},t,\Phi_{t,\wb{t}}(z))+
d\Phi_{\wb{t},t}(\Phi_{t,\wb{t}}(z),f(t,\Phi_{t,\wb{t}}(z)))=0
\end{align*}
as a special case of~(\ref{tool-flw}).
\end{proof}
\section{Initial value problems in Fr\'{e}chet manifolds}\label{sec-now-mfd}
We now extend the local theory of Section~\ref{sec-ex-uni}
to the case of Fr\'{e}chet manifolds.
Also, we prove Proposition~\ref{thmA}
and Theorem~\ref{thmB}.
The treatment emulates the discussion
of existence and uniqueness of $C^1$-solutions
in~\cite[\S2.5]{GaN}.
\begin{defn}\label{def-ode-mfd}
Let $M$ be a $C^1$-manifold modelled on a Fr\'{e}chet space~$E$
and $f\colon W\to TM$ be a function on a subset
$W\sub \R\times M$ such that $f(t,y)\in T_yM$
for all $(t,y)\in W$.
We say that a function $\gamma\colon I\to M$
on a non-degenerate interval $I\sub \R$
is a \emph{Carath\'{e}odory solution}
to the differential equation
\begin{equation}\label{ode-mfd}
\dot{y}(t)=f(t,y(t))
\end{equation}
if $\gamma$ is absolutely continuous,
$(t,\gamma(t))\in W$ for all $t\in I$
and
\[
\dot{\gamma}(t)=f(t,\gamma(t))
\]
for $\lambda_1$-almost all $t\in I$
(using notation as in \ref{dotty}).
If $(t_0,y_0)\in W$ is given and $\gamma$ satisfies,
moreover, the condition $\gamma(t_0)=y_0$,
then $\gamma$ is called a Carath\'{e}odory
solution to the initial value problem
\begin{equation}\label{ivp-mfd}
\dot{y}(t)=f(t,y(t)),\quad y(t_0)=y_0.
\end{equation}
\end{defn}
\begin{rem}\label{mfdlocalize}
If $\phi\colon U_\phi\to V_\phi\sub E$
is a chart for~$M$ in the situation of Definition~\ref{def-ode-mfd},
we define a function
\[
f_\phi\colon W_\phi\to E,\quad (t,y)\mto d\phi(f(t,\phi^{-1}(y)))
\]
on $W_\phi:=\{(t,y)\in \R\times V_\phi \colon (t,\phi^{-1}(y))\in W\}\sub\R\times E$.
Let $\gamma\colon I\to M$ be a continuous function on a non-degenerate
interval $I\sub\R$ such that $(t,\gamma(t))\in W$ for all
$t\in I$.
Then~$\gamma$ is a Carath\'{e}odory
solution to (\ref{ode-mfd})
if and only if $\phi\circ \gamma|_K$
is a Carath\'{e}odory solution to
\begin{equation}\label{incha-dgl}
y'(t)=f_\phi(t,\phi^{-1}(y(t)))
\end{equation}
for each chart $\phi\colon U_\phi\to V_\phi\sub E$
and each non-degenerate subinterval $K\sub I$
such that $\gamma(K)\sub U_\phi$.
The latter holds if and only if, for each $t_0\in I$,
there are a chart~$\phi$ and a subinterval $K\sub I$
which is a neighbourhood of~$t_0$ in~$I$ such that
$\gamma(K)\sub U_\phi$ and $\phi\circ\gamma|_K$ solves
(\ref{incha-dgl}),
due to~\ref{C1actsAC}.
If $I=[a,b]$,
equivalently we may take~$K$
in finite set of subintervals
which cover~$I$.
\end{rem}
\begin{defn}\label{def-locuni-mfd}
In the situation of Definition~\ref{def-ode-mfd},
we say that the differential equation (\ref{ode-mfd})
satisfies \emph{local uniqueness of Carath\'{e}odory solutions}
if the following holds:
For all Carath\'{e}odory solutions
$\gamma_1\colon I_1\to M$ and $\gamma_2\colon I_2\to M$
of (\ref{ode-mfd}) and $t_0\in I_1\cap I_2$
such that $\gamma_1(t_0)=\gamma_2(t_0)$,
there exists an interval $K\sub\R$
which is an open neighbourhood of~$t_0$
in $I_1\cap I_2$ such that
$\gamma_1|_K=\gamma_2|_K$.
\end{defn}
\begin{la}\label{mfd-loc-glob}
Consider a differential equation {\rm(\ref{ode-mfd})}
as in Definition~{\rm\ref{def-ode-mfd}},
which satisfies local uniqueness of Carath\'{e}odory
solutions. Assume that
$\gamma_1\colon I_1\to M$ and $\gamma_2\colon I_2\to M$
are Carath\'{e}odory solutions to {\rm(\ref{ode-mfd})}
such that $\gamma_1(t_0)=\gamma_2(t_0)$ for some
$t_0\in I_1\cap I_2$. Then
$\gamma_1|_{I_1\cap I_2}=\gamma_2|_{I_1\cap I_2}$.
\end{la}
\begin{proof}
We can repeat the proof of Lemma~\ref{la-loc-glob}
with $M$ in place of~$E$.
\end{proof}
\begin{defn}\label{def-locex-mfd}
Let $M$ be a $C^1$-manifold modelled on a Fr\'{e}chet space,
$J\sub\R$ be a non-degenerate interval
and $f\colon W\to TM$ be a function on a subset
$W\sub J\times M$
such that $f(t,y)\in T_yM$ for all $(t,y)\in W$.
We say that the differential equation (\ref{ode-mfd})
satisfies \emph{local existence of Carath\'{e}odory
solutions}
if for all $(t_0,y_0)\in W$,
there exists a Carath\'{e}odory solution
$\gamma\colon I\to M$ to
the initial value problem (\ref{ivp-mfd})
such that $I$ is a
relatively open subinterval of~$J$.
\end{defn}
\begin{la}\label{max-sol}
Let $M$ be a $C^1$-manifold modelled on a Fr\'{e}chet space,
$J\sub\R$ be a non-degenerate interval
and $f\colon W\to TM$ be a function on a subset
$W\sub J\times M$
such that $f(t,y)\in T_yM$ for all $(t,y)\in W$.
Assume that the differential equation {\rm(\ref{ode-mfd})}
satisfies both local existence of Carath\'{e}odory
solutions and local uniqueness.
Then, for all $(t_0,y_0)\in W$,
there exists a Carath\'{e}odory
solution $\gamma\colon I\to M$ to the initial value problem
{\rm(\ref{ivp-mfd})}
such that $I_\eta\sub I$ and $\eta=\gamma|_{I_\eta}$
for each Carath\'{e}odory solution $\eta\colon I_\eta\to M$
to {\rm(\ref{ivp-mfd})}.
Moreover, $I$ is relatively open in~$J$.
\end{la}
\begin{proof}
Case~1: Let us first assume that $t_0$ is the minimum of~$J$
(Case~2, that $t_0$ is the maximum of~$J$, is analogous).
The set $L$ of all
$\tau\in \; ]t_0,\infty[\,\cap J$
such that (\ref{ivp-mfd})
has a solution $\eta_\tau\colon [t_0,\tau]\to M$
is a subinterval of $\,]t_0,\infty[\,\cap J$,
and also $I:=L\cup\{t_0\}$ is an interval.
If $t\in I$, there exists $\tau\in L$ such that $t\in [t_0,\tau]$;
we define
\[
\gamma(t):=\gamma_\tau(t).
\]
If $\sigma,\tau\in L$ and $\sigma\leq \tau$,
then $\gamma_\sigma=\gamma_\tau|_{[t_0,\sigma]}$
by Lemma~\ref{mfd-loc-glob},
entailing that $\gamma\colon I\to M$ is well defined.
By construction, we have
\[
\gamma|_{[t_0,\tau]}=\gamma_\tau
\]
for each $\tau\in L$; notably, $\gamma(t_0)=\gamma_\tau(t_0)=y_0$.
If $[a,b]\sub I$ with $a<b$,
then $\gamma|_{[a,b]}=\gamma_b|_{[a,b]}$ is absolutely
continuous, whence $\gamma$ is absolutely continuous.\\[2.3mm]
If $I$ does not have a maximum,
then $I$ is open in $J=[t_0,\infty[\,\cap J$.
We can take an ascending sequence $t_0<\tau_1<\tau_2<\cdots$ tending to
the supremum of~$I$. Since $\gamma|_{[t_0,\tau_j[}$
is differentiable at almost all $t$ in its domain,
with derivative $f(t,\gamma(t))$,
the same is true of~$\gamma$.
Thus~$\gamma$ is a Carath\'{e}odory solution to~(\ref{ivp-mfd}).\\[2.3mm]
If $I$ has a maximum~$\tau$, then $\tau\in L$ and $\gamma=\gamma_\tau$
is a Carath\'{e}odory solution to (\ref{ivp-mfd}).
We show that $\tau$ is also the maximum of~$J$,
whence $[t_0,\tau]=J$ is open in~$J$.
If not, using local existence we find a
Carath\'{e}odory solution $\eta$ of (\ref{ode-mfd})
on an interval $K\sub J$ with $\tau$
in the interior of~$K$ relative~$J$, such that $\eta(\tau)=\gamma(\tau)$.
Then $[\tau,\theta]\sub K$ for some $\theta>\tau$
and we can extend~$\gamma$ to a solution of (\ref{ivp-mfd})
defined on $[t_0,\theta]$ by taking $t\in [\tau,\theta]$
to $\eta(t)$. Thus $\theta\in L$; since $\theta>\tau$,
this contradicts $\tau=\max L$.\\[2.3mm]
Case~3: If $t_0$ is in the interior
of~$J$ relative~$\R$,
then Case~1 provides a solution
$\gamma_+$ to the initial value problem
on a largest subinterval $I_+\sub J\cap [t_0,\infty[$.\linebreak

\noindent
Likewise, Case~2 
provides a solution
$\gamma_-$
on a largest subinterval $I_-$ of\linebreak
$J\cap \,]{-\infty},t_0]$.
Then $I:=I_+\cup I_-$
and the function $\gamma\colon I\to M$ which is defined
piecewise via $\gamma(t):=\gamma_\pm(t)$ for $t\in I_\pm$
are as required.
\end{proof}
\begin{defn}\label{max-flow}
The solution $\gamma_{t_0,y_0}:=\gamma$
described in Lemma~\ref{max-sol}
is called the \emph{maximal solution}
to the initial value problem (\ref{ivp-mfd});
we write $I_{t_0,y_0}:=I$ for its domain.
We abbreviate
\[
\Omega:=\bigcup_{(t_0,y_0)\in W}I_{t_0,y_0}\times\{(t_0,y_0)\}\;\sub\,
\R\times \R\times M
\]
and call
\[
\Fl\colon \Omega\to M,\quad \Fl(t,t_0,y_0):=\gamma_{t_0,y_0}(t)
\]
the \emph{maximal flow} associated with (\ref{ode-mfd}).
We also write $\Fl_{t,t_0}(y_0):=\Fl(t,t_0,y_0)$.
\end{defn}
\begin{defn}\label{def-loc-flo-mfd}
Let $J\sub \R$ be a non-degenerate interval, $k\in\N\cup\{\infty\}$,
$M$ be a $C^k$-manifold
modelled on a Fr\'{e}chet space
and $f\colon W\to TM$ be a function on an open subset $W\sub J\times M$
such that $f(t,y)\in T_yM$ for all $(t,y)\in W$.
We say that the differential equation $\dot{y}(t)=f(t,y(t))$
\emph{admits local flows which are pullbacks of $C^k$-maps}
if,
for all $(\wb{t},\wb{y}) \in W$,
there exist a relatively open interval $I\sub J$ with $\wb{t}\in I$,
an open neighborhood~$V$ of $\wb{y}$ in~$M$
with $I\times V\sub W$
and function
\[
\Phi\colon I\times I\times V\to M
\]
with the following properties:
\begin{itemize}
\item[(a)]
For all $(t_0,y_0)\in I\times V$, the function
$I\to E$, $t\mto \Phi_{t,t_0}(y_0):=\Phi(t,t_0,y_0)$ is a Carath\'{e}odory
solution to
the initial value problem (\ref{ivp-mfd});
\item[(b)]
There is an open $\wb{y}$-neighbourhood $Y\sub V$
such that $\Phi_{t_1,t_0}(Y)\sub V$ for all $t_0,t_1\in I$ and
\[
\Phi_{t_2,t_1}(\Phi_{t_1,t_0}(y_0))=\Phi_{t_2,t_0}(y_0)
\;\;\mbox{for all $\,t_0,t_1, t_2\in I$ and $y_0\in Y$;}
\]
\item[(c)]
There exist $C^k$-manifolds~$N_1$ and $N_2$
modelled on Fr\'{e}chet spaces
$E_1$ and $E_2$, respectively,
absolutely continuous functions
$\zeta_1\colon I\to N_j$ for $j\in \{1,2\}$
and a $C^k$-map $\Psi\colon N_1\times N_2 \times V\to M$
such that\footnote{In particular, $\Phi$ is continuous.}
\[
\Phi(t,t_0,y_0)=\Psi(\zeta_1(t),\zeta_2(t_0),y_0)\quad\mbox{for all $(t,t_0,y_0)
\in I\times I\times V$.}
\]
Moreover, we require the existence of a Borel set $I_0\sub I$
with $\lambda_1(I\setminus I_0)=0$ such that
$\frac{d}{dt}\Phi_{t,t_0}(y_0)$ exists
and
\[
\frac{d}{dt}\Phi_{t,t_0}(y_0)=f(t,\Phi_{t,t_0}(y_0))
\]
for all $y_0\in V$, $t_0\in I$ and $t \in I_0$.
\end{itemize}
\end{defn}
\begin{rem}
If $\dot{y}(t)=f(t,y(t))$ (as in Definition~\ref{def-loc-flo-mfd})
admits local flows which are pullbacks
of $C^k$-maps, then $\dot{y}(t)=f(t,y(t))$ satisfies
local existence of Carath\'{e}odory solutions
(arguing as in Remark~\ref{flo-then-ex}).
\end{rem}
\begin{rem}
We might speak about
local flows which are pullbacks of $C^k$-maps
\emph{on Fr\'{e}chet manifolds} in the situation
of Definition~\ref{def-loc-flo-mfd},
and speak about local flows which are pullbacks
of $C^k$-maps
\emph{on open subsets of Fr\'{e}chet spaces}
in the situation of Definition~\ref{defn-loc-flo},
to distinguish clearly between the concepts
(likewise, we should use separate terminology
in Definition~\ref{defn-loc-flo-2}). But it will
always be clear from the context
what is intended.
\end{rem}
{\bf Proof of Proposition~\ref{thmA}.}
For each chart $\phi\colon U_\phi\to V_\phi\sub E$ of~$M$,
let $f_\phi\colon W_\phi\to E$ be as in~Remark~\ref{mfdlocalize}.
We claim: Each differential equation
\begin{equation}\label{loca-dgl}
y'(t)\, =\, f_\phi(t,y(t))
\end{equation}
admits local flows which are pullbacks of~$C^1$-maps,
in the sense of Definition~\ref{defn-loc-flo}.
As a consequence, each of the differential equations
(\ref{loca-dgl})
satisfies local uniqueness of Carath\'{e}odory
solutions, by Lemma~\ref{locflow-uniq}. This implies that (\ref{ode-mfd})
satisfies local uniqueness of Carath\'{e}odory solutions (cf.\ Remark~\ref{mfdlocalize}),
which completes the proof.
To establish the claim, let $(\bar{t},\bar{z})\in W_\phi$
and $\bar{y}:=\phi^{-1}(\bar{z})\in U_\phi$.
Then
\[
(\bar{t},\bar{y})\in W\sub J\times M.
\]
Let $I$, $V$, $\Phi$, $Y$, $N_1$, $N_2$, $E_1$, $E_2$,
$\Psi$, $\zeta_1$, $\zeta_2$, and $I_0$
be as in Definition~\ref{def-loc-flo-mfd}.
After shrinking~$I_0$ if necessary, we may assume
that, moreover, $\dot{\zeta}_1(t)$ and $\dot{\zeta}_2(t)$
exist for all $t\in I_0$.\\[2.3mm]
There exist charts $\phi_j\colon U_j\to V_j\sub E_j$
of $N_j$ with $\zeta_j(\wb{t})\in U_j$
for $j\in\{1,2\}$.
Since $\Psi(\zeta_1(\bar{t}),\zeta_2(\bar{t}),\bar{y})
=\Phi_{\bar{t},\bar{t}}(\bar{y})=\bar{y}\in U_\phi$
and $\Psi$ is continuous,
after shrinking $U_1$ and $U_2$
we find an open $\bar{y}$-neighbourhood $U\sub U_\phi$
such that
\[
\Psi(U_1\times U_2\times U)\sub U_\phi.
\]
Then $P:=\phi(U)$ is an open $\bar{z}$-neighbourhood
in~$V_\phi$.
After shrinking $U_1$ and $U_2$ further,
we may assume that
\[
\Psi(U_1\times U_2\times Z)\sub U
\]
for some open $\bar{y}$-neighbourhood $Z\sub U$.
Then $Q:=\phi(Z)$ is an open $\bar{z}$-neighbourhood
in~$P$.
After shrinking~$I$, we may assume that
$\zeta_j(I)\sub U_j$ for $j\in\{1,2\}$.
Then
\[
\Psi_\phi\colon V_1\times V_2\times P\to E,\quad
(u,v,w)\mto \phi(\Psi(\phi_1^{-1}(u),\phi_2^{-1}(v),\phi^{-1}(w)))
\]
is a $C^1$-map. Moreover, $\alpha:=\phi_1\circ\zeta_1\colon I\to V_1\sub E_1$
and $\beta:=\phi_2\circ \zeta_2\colon I\to V_2\sub E_2$
are absolutely continuous functions which are differentiable
at each $t\in I_0$.
Define
\[
\Phi_\phi\colon I\times I\times P\to E,\quad
(t,t_0,z_0)\mto \Psi_\phi(\alpha(t),\beta(t_0),z_0)
=\phi(\Phi_{t,t_0}(\phi^{-1}(z_0))).
\]
Using \ref{C1onpw},
it is now easy to check that $\Phi_\phi$ and $\Psi_\phi$
in place of $\Phi$ and $\Psi$,
with $P$ and $Q$ in place of $V$ and $Y$,
satisfy the conditions (a)--(c) of Definition~\ref{defn-loc-flo},
for~$f_\phi$ in place of~$f$ and~$\bar{z}$ in place of~$\bar{y}$.
This establishes the claim. $\;\square$
\section{Initial value problems on {\boldmath$G$}-manifolds}
We now prove Theorem~\ref{thmB}, which provides
a criterion for
local uniqueness of Carath\'{e}odory
solutions to differential equations
on smooth $G$-manifolds which are given by time-dependent
fundamental vector fields.\\[2.3mm]
{\bf Proof of Theorem~\ref{thmB}.}
Abbreviate $I:=J:=[a,b]$.
Define
\begin{equation}\label{therelf}
f\colon I\times M\to TM,\quad f(t,y):=\gamma(t)_\sharp(y).
\end{equation}
Since $\eta\colon I \to G$ is absolutely continuous
and the evolution of $\gamma$,
there exists a Borel set $I_0\sub I$
with $\lambda_1(I\setminus I_0)=0$
such that $\dot{\eta}(t)$
exists for all $t\in I_0$
and
\[
\dot{\eta}(t)=\eta(t).\gamma(t)\quad\mbox{for all $\,t\in I_0$.}
\]
Abbreviate $\sigma_y(g):=y.g:=\sigma(y,g)$ for $y\in M$ and $g\in G$.
We define
\[
\Phi\colon I\times I\times M\to M,\quad (t,t_0,y_0)\mto \sigma(y_0,\eta(t_0)^{-1}\eta(t))
\]
and write $\Phi_{t,t_0}(y_0):=\Phi(t,t_0,y_0)$.
The map
\[
\Psi\colon G\times G\times M\to M,\quad
(g,h,y)\mto \sigma(y,h^{-1}g)
\]
is smooth and
\[
\Phi(t,t_0,y_0)=\Psi(\zeta_1(t),\zeta_2(t_0),y_0)
\]
for all $(t,t_0,y_0)\in I\times I\times M$
with $\zeta_1:=\zeta_2:=\eta$.\\[2.3mm]
Given $y_0\in M$ and $t_0\in I$, we obtain an absolutely
continuous function $\theta\colon I\to G$ via
\[
\theta(t):=\eta(t_0)^{-1}\eta(t).
\]
Also
$\phi\colon I\to M$, $\phi(t):=y_0.\theta(t)=\Phi_{t,t_0}(y_0)$
is absolutely continuous as $\phi(t)=\sigma_{y_0}\circ\theta$.
By definition, $\phi(t_0)=y_0$.
For each $t\in I_0$, we have
\[
\dot{\theta}(t)=\eta(t_0)^{-1}.\dot{\eta}(t)=\eta(t_0)^{-1}\eta(t).\gamma(t),
\]
whence $\theta(t)^{-1}.\dot{\theta}(t)=
\eta(t)^{-1}\eta(t_0).(\eta(t_0)^{-1}\eta(t).\gamma(t))=\gamma(t)$ and
\begin{eqnarray*}
\dot\phi(t)&=&\frac{d}{ds}\Big|_{s=0}\phi(t+s)=\frac{d}{ds}\Big|_{s=0}y_0.\theta(t+s)\\
&=& \frac{d}{ds}\Big|_{s=0}
(y_0.\theta(t)).\theta(t)^{-1}\theta(t+s)
=
T(\sigma_{y_0.\theta(t)})(\theta(t)^{-1}.\dot{\theta}(t))\\
&=& T(\sigma_{y_0.\theta(t)})\gamma(t)
= \gamma(t)_\sharp(\phi(t))=f(t,\phi(t)).
\end{eqnarray*}
Hence $\phi$ solves (\ref{ivp-mfd}).
Thus conditions~(a) and~(c)
of Definition~\ref{def-loc-flo-mfd}
are satisfied (with $I:=J$, $N_1:=N_2:=G$,
$Y:=V:=M$, $\Phi$, $\Psi$,
$\zeta_1$, $\zeta_2$, and $I_0$ independent of $(\wb{t},\wb{y})\in [a,b]\times M$).
For all
$t_0,t_1,t_2\in I$ and $y_0\in M$, setting
$y_1:=\Phi_{t_1,t_0}(y_0)=y_0.\eta(t_0)^{-1}\eta(t_1)$,
we have
\begin{eqnarray*}
\Phi_{t_2,t_1}(\Phi_{t_1,t_0}(y_0))\!\!&\!\!=\!\! &\!\! \Phi_{t_2,t_1}(y_1)
=y_1.\eta(t_1)^{-1}\eta(t_2)
=(y_0.\eta(t_0)^{-1}\eta(t_1)).\eta(t_1)^{-1}\eta(t_2)\\
\!\!&\!\!=\!\! &\!\!  y_0.\eta(t_0)^{-1}\eta(t_2)=\Phi_{t_2,t_0}(y_0).
\end{eqnarray*}
Thus condition~(b) of Definition~\ref{def-loc-flo-mfd}
is satisfied by~$\Phi$.
We have shown that (\ref{ode-mfd}),
applied to~$f$ as in (\ref{therelf}),
admits local flows which are pullbacks of $C^\infty$-maps.
Thus (\ref{ode-mfd}) satisfies local uniqueness of Carath\'{e}odory
solutions. As $t\mto\Phi_{t,t_0}(y_0)$
is a solution to~(\ref{ivp-mfd}) defined on all of~$I$,
we see that $I_{t_0,y_0}=I$
for all $(t_0,y_0)\in I\times M$
and $\gamma_{t_0,y_0}(t)=\Phi_{t,t_0}(y_0)$.
The domain of the maximal flow~$\Fl$ of~(\ref{ode-mfd})
is therefore given by $\Omega=I\times I\times M$
here, and $\Fl(t,t_0,y_0)=\gamma_{t_0,y_0}(t)=\Phi_{t,t_0}(y_0)$,
which completes the proof.
$\;\square$
\section{Reachable neighbourhoods in {\boldmath$G$}-manifolds}\label{sec-approachable}
In this section, we prove Theorem~\ref{thmC}.
We begin with preparatory results.
First, we discuss the approximation
of vector-valued $\cL^1$-functions by
staircase functions.
Using the Trotter product formula,
we then provide preparations enabling
us to replace controls in the convex hull
$\conv\{v_1,\ldots, v_m\}$ of finitely many vectors
with controls in $\{v_1,\ldots,v_m\}$.\\[2.3mm]
For each Fr\'{e}chet space~$E$,
the space of $E$-valued staircase functions
is dense in $\cL^1([a,b],E)$.
Moreover, we have:
\begin{la}\label{approx-stair}
Let $a<b$ be real numbers, $E$ be a Fr\'{e}chet space,
and $\gamma\in \cL^1([a,b],E)$.
Let $q$ be a continuous seminorm on~$E$ and $\ve>0$.
Then we~have:
\begin{itemize}
\item[\rm(a)]
There exists a staircase function
$\eta\colon [a,b]\to E$ such that $\eta([a,b])\sub \gamma([a,b])$
and $\|\gamma-\eta\|_{\cL^1,q}\leq\ve$.
\item[\rm(b)]
There exists a continuous function
$\theta\colon [a,b]\to E$ such that $\theta([a,b])\sub \conv(\gamma([a,b]))$
and $\|\gamma-\theta\|_{\cL^1,q}\leq\ve$.
\end{itemize}
\end{la}
\begin{proof}
(a) Let $(y_n)_{n\in\N}$ be a sequence in~$\gamma([a,b])$
such that $\{y_n\colon n\in\N\}$ is dense in $\gamma([a,b])$.
Abbreviate $T:=b-a$. Define
\[
A_1:=\Big\{t\in [a,b]\colon q(\gamma(t)-y_1)<\frac{\ve}{6T}\Big\}
\]
and, recursively,
\[
A_n:=\Big\{t\in [a,b]\setminus (A_1\cup\cdots\cup A_{n-1})\colon q(\gamma(t)-y_n)<\frac{\ve}{6T}\Big\}
\]
for integers $n\geq 2$. Then $(A_n)_{n\in\N}$
is a sequence of pairwise disjoint Borel sets with union $[a,b]$.
There exists $N\in\N$ such that $R:=\bigcup_{n>N}A_n$ satisfies
\[
\int_R q(\gamma(t))\,dt<\ve/6
\]
and
\begin{equation}\label{auchnoch}
q(y_1)\lambda_1(R)<\ve/12.
\end{equation}
Using characteristic functions of the sets $A_n$, we define
\[
\gamma_1\colon [a,b]\to E,\quad
t\mto\sum_{n=1}^N\one_{A_n}(t)y_n.
\]
If $n\in\{1,\ldots, N\}$ and $t\in A_n$, then
$q(\gamma(t)-\gamma_1(t))=q(\gamma(t)-y_n)<\ve/(6T)$.
If $t\in R$, then $q(\gamma(t)-\gamma_1(t))=q(\gamma(t))$.
Hence, abbreviating $A:=A_1\cup\cdots\cup A_N$,
\[
\|\gamma-\gamma_1\|_{\cL^1,q}=\int_Aq(\gamma(t)-\gamma_1(t))\,dt
+\int_Rq(\gamma(t))\,dt\leq
\frac{\ve}{6T}\lambda_1(A)+\frac{\ve}{6}\leq \frac{\ve}{3}.
\]
Pick a real number $C>0$ such that
\[
q(y_n)\leq C\mbox{ for all $n\in\{1,\ldots, N\}$.}
\]
By inner regularity of $\lambda_1$,
for $n\in\{1,\ldots, N\}$ we find a compact subset $K_n\sub A_n$
such that
\begin{equation}\label{willct}
\lambda_1(A_n\setminus K_n)\leq \frac{\ve}{12NC}.
\end{equation}
Hence $\gamma_2:=\sum_{n=1}^N\one_{K_n}y_n$ satisfies
\[
\|\gamma_1-\gamma_2\|_{\cL^1,q}
=\sum_{n=1}^Nq(y_n)\lambda_1(A_n\setminus K_n)\leq\ve/12\leq \ve/3.
\]
The compact sets $K_1,\ldots, K_N$ are pairwise disjoint,
whence we find pairwise disjoint open subsets $U_1,\ldots, U_N$
of $[a,b]$ such that $K_n\sub U_n$ for all $n\in\{1,\ldots, N\}$.
By outer regularity of~$\lambda_1$, after shrinking
the sets if necessary we may assume that
\[
\lambda_1(U_n\setminus K_n)<\frac{\ve}{6NC}\mbox{ for all $n\in\{1,\ldots,N\}$.}
\]
After replacing $U_n$ with a finite number
of its connected components, we may assume that each $U_n$
is a union of finitely many pairwise disjoint intervals
which are open in $[a,b]$.
Let
\[
U:=\bigcup_{n=1}^NU_n\quad\mbox{and}\quad B:=[a,b]\setminus U.
\]
Then $\eta:=y_1\one_B+\sum_{n=1}^N y_n\one_{U_n}$
is a staircase function with $\eta([a,b])\sub \gamma([a,b])$
since $\eta(t)=y_n\in \gamma([a,b])$
for each $n\in\{1,\ldots,N\}$ and $t\in U_n$, while
$\eta(t)=y_1\in \gamma([a,b])$ for all $t\in B$.
Since
\[
B\sub [a,b]\setminus (K_1\cup\cdots\cup K_N)
=R\cup\bigcup_{n=1}^N(A_n\setminus K_n),
\]
using (\ref{auchnoch}) and (\ref{willct}) we estimate
\[
q(y_1)\lambda_1(B)\leq \ve/12+\ve/12=\ve/6.
\]
Thus
\[
\|\gamma_2-\eta\|_{\cL^1,q}
=\sum_{n=1}^N\lambda_1(U_n\setminus K_n)q(y_n)+\lambda_1(B)q(y_1)
\leq \ve/6 + \ve/6=\ve/3,
\]
whence
$\|\gamma-\eta\|_{\cL^1,q}\leq\|\gamma-\gamma_1\|_{\cL^1,q}
+\|\gamma_1-\gamma_2\|_{\cL^1,q}+\|\gamma_2-\eta\|_{\cL^1,q}\leq\ve$.\\[2.3mm]
(b) By (a), there exists a staircase function $\eta\colon [a,b]\to E$
such that $\eta([a,b])\sub\gamma([a,b])$ and $\|\gamma-\eta\|_{\cL^1,q}\leq\ve/2$.
There exist $n\in \N$ and numbers $a=t_0<t_1<\cdots<t_n=b$
such that $\eta|_{]t_{j-1},t_j[}$ is constant,
with value $y_j\in\gamma([a,b])$, for all $j\in\{1,\ldots,n\}$.
Choose $\delta>0$ so small that $2\delta<t_j-t_{j-1}$ for all $j\in\{1,\ldots,n\}$.
For $k\in \N$, we define $\theta_k\in C([a,b],E)$ piecewise,
as follows: We let $\theta_k(t):=y_1$ for $t\in [a,t_1-\delta/k]$
and $\theta_k(t):=y_n$ for $t\in [t_{n-1}+\delta/k,b]$.
We let $\theta_k(t):=y_j$ for all $j\in \{2,\ldots,n-1\}$
and $t\in [t_{j-1}+\delta/k,t_j-\delta/k]$.
Finally, for $j\in \{1,\ldots,n-1\}$ and $t\in [t_j-\delta/k,t_j+\delta/k]$,
we define
\[
\theta_k(t):=y_j+\frac{t-t_j+\delta/k}{(2\delta)/k}(y_{j+1}-y_j);
\]
thus $\theta_k|_{[t_j-\delta/k,t_j+\delta/k]}$
is a restriction of the unique affine-linear map
taking $t_j-\delta/k$ to $y_j$ and $t_j+\delta/k$ to $y_{j+1}$.
Note that the image of $\theta_k$ is contained in the convex hull
$C$ of $\eta([a,b])$ and hence in the convex hull of~$\gamma([a,b])$.
As $k\to\infty$, we have $\theta_k(t)\to\eta(t)$
for all $t\in [a,b]\setminus\{t_0,\ldots, t_n\}$
and thus for almost all $t\in [a,b]$.
Since~$C$ is bounded, also $q(C)$ is bounded.
The constant function $g\colon [a,b]\to [0,\infty[$,
$t\mto 2\sup q(C)$ is $\lambda_1$-integrable
and $q(\eta(t)-\theta_k(t))\leq g(t)$
for all $t\in[a,b]$. Thus
\[
\|\eta-\theta_k\|_{\cL^1,q}=\int_a^bq(\eta(t)-\theta_k(t))\,dt\to 0
\]
as $k\to\infty$ by dominated convergence. Notably,
we find $k\in\N$ such that $\theta:=\theta_k$ satisfies
$\|\eta-\theta\|_{\cL^1,q}\leq\ve/2$. Then $\|\gamma-\theta\|_{\cL^1,q}\leq\ve$.
\end{proof}
\begin{rem}\label{appr-lcx}
If $E$ may not be a Fr\'{e}chet space,
but is an arbitrary locally convex space,
and $\gamma\colon [a,b]\to E$
is a piecewise continuous function,
then all conclusions of Lemma~\ref{approx-stair}
remain valid, with identical proof.\footnote{Note that $\zeta([\alpha,\beta])$
is compact and metrizable
(using \cite[Theorem~4.4.17]{Eng}),
whence $\zeta([\alpha,\beta])$
and its subsets are metrizable and separable,
for all real numbers $\alpha<\beta$
and each continuous function $\zeta\colon [\alpha,\beta]\to E$.
Hence $\gamma-\eta$ and $\gamma-\theta$
are Borel measurable also in the current
variant of Lemma~\ref{approx-stair},
and so are the functions
$\gamma-\gamma_1$, $\gamma_1-\gamma_2$, $\gamma_2-\eta$,
and $\eta-\theta_k$ in its proof (by Lemma~\ref{compo-prod}).}
\end{rem}
We shall use the following concept (cf.\ \cite[Definition~14.3]{SEM}).\footnote{If $R\geq 1$
there, we may replace $R$ with~$1$; if $R\leq 1$, replace
$p$ with $p/R$ and $R$ with $1$.}
\begin{defn}\label{locmu}
A Lie group $G$ modelled on a locally convex space~$E$
is called \emph{locally $\mu$-convex}
if there exists a chart $\phi\colon U_\phi\to V_\phi\sub E$ of~$G$
with $e\in U_\phi$ and $\phi(e)=0$ with the following property:
For each continuous seminorm~$q$ on~$E$,
there exist a continuous seminorm~$p$ on~$E$
such that
\[
g_1g_2\cdots g_n\in U_\phi\quad\mbox{and}\quad
q(\phi(g_1\cdots g_n))\leq\sum_{j=1}^np(\phi(g_j))
\]
for all $n\in \N$ and all $g_1,\ldots, g_n\in U_\phi$ such that
$\sum_{j=1}^np(\phi(g_j))< 1$.
\end{defn}
Then every chart taking $e$ to $0$ has this property
(see \cite[Remark 14.4]{SEM}).\\[2.3mm]
Let $X$ be a set.
If $\gamma\colon [0,a]\to X$ and $\eta\colon [0,b]\to X$ are
staircase functions,
we write $\gamma*\eta\colon [0,a+b]\to X$ for the concatenation
defined via $t\mto\gamma(t)$ for $t\in [0,a]$,
$t\mto \eta(t-a)$ for $t\in\,]a,a+b]$.
\begin{la}\label{Trotter-use}
Let $G$ be a $C^0$-regular Lie group modelled on a locally convex space,
with Lie algebra
$\cg$. Let $m\in\N$, $t_1,\ldots, t_m\in \,]0,\infty[$,
$T:=t_1+\cdots+t_m$,
$v_1,\ldots, v_m\in\cg$, and $W$ be a neighbourhood
of
\[
g:=\exp_G(t_1v_1+\cdots+ t_mv_m)
\]
in~$G$. Then there exists a staircase function\footnote{If
$\gamma(t)=(\gamma_1*\cdots*\gamma_k)(t)$ except for finitely
many $t$, with constant functions
$\gamma_1,\ldots,\gamma_k$, then
$\evol(\gamma):=\evol(\gamma_1)\cdots\evol(\gamma_k)$.
See Remark~\ref{basics-pw-evol} for more details.}
$\gamma\colon [0,T]\to \cg$ with $\gamma([0,T])\sub\{v_1,\ldots,v_m\}$
such that $\evol(\gamma)\in W$.
\end{la}
\begin{proof}
The proof is by induction on $m\in\N$.
The case $m=1$ is trivial, as $\exp_G(t_1v_1)=\evol(\gamma)$
for the constant function $\gamma\colon [0,t_1]\to\cg$, $t\mto v_1$,
which is a staircase function. Now let $m>1$ and assume
the assertion holds for $m-1$ in place of~$m$.
Abbreviate $v:=t_1v_1+\cdots+t_{m-1}v_{m-1}$.
Since $G$ is $C^0$-regular, it is locally $\mu$-convex
(see \cite{Han}), whence $G$ satisfies the strong Trotter
property formulated in~\cite[p.\,7]{MeR} (see \cite{Tro}).\footnote{For the case
of $L^1$-regular Fr\'{e}chet-Lie groups
considered in Theorem~\ref{thmC}, the strong Trotter property was
already established in~\cite[Theorem~I]{MeR}.}
In particular, the Trotter product formula holds
for each pair of Lie algebra elements, and thus
\[
g=\exp_G(v+t_mv_m)=\lim_{n\to\infty}\Big(\exp_G(v/n)\exp_G(t_mv_m/n)\Big)^n.
\]
We therefore find $n\in\N$ such that
\[
\Big(\exp_G(v/n)\exp_G(t_mv_m/n)\Big)^n\in W.
\]
Now $\exp_G(t_mv_m/n)=\evol(\theta)$
for the constant function $\theta\colon [0,t_m/n]\to\cg$, $t\mto v_m$.
There exists an open neighbourhood~$U$ of $\exp_G(v/n)$ in~$G$ such that
\[
(u\evol(\eta))^n\in W\quad\mbox{for all $\,u\in U$.}
\]
Since $\exp_G(v/n)=\exp_G(t_1v_1/n+\cdots+t_{m-1}v_{m-1}/n)\in U$,
we have
\[
\evol(\eta)\in U
\]
for a staircase function $\eta\colon [0,(T-t_m)/n]\to\cg$
with image in $\{v_1,\ldots, v_{m-1}\}$, by induction.
If we define
\[
\gamma:=(\eta*\theta)*\cdots *(\eta*\theta)
\]
as the concatenation of $n$ copies of $\eta*\theta$,
then
\[
\evol(\gamma)=(\evol(\eta)\evol(\theta))^n\in W.
\]
Moreover, $\gamma$ is a staircase function
with image in $\{v_1,\ldots,v_m\}$.
\end{proof}
\noindent
{\bf Proof of Theorem~\ref{thmC}.}
(b) implies (a) since $\cL^p([0,T],\cg)\sub\cL^1([0,T],\cg)$.
Likewise,
(d) implies (c) as staircase functions are piecewise
continuous, and (c) implies~(b)
as piecewise continuous functions $[0,T]\to\cg$
are in $\cL^p([0,T],\cg)$.\\[2.3mm]
(a) implies (d): Let $\gamma\in \cL^1([0,T],\cg)$
such that $x_0.\evol(\gamma)\in U$.
Since $\evol\colon \cL^1([0,T],\cg)\to G$
and the action $M\times G\to M$ are continuous,
there exists $\ve>0$ and a continuous seminorm~$P$ on
$\cg$ such that
\[
x_0.\evol(\eta)\in U \mbox{ for all $\eta\in \cL^1([0,T],\cg)$ such that
$\|\gamma-\eta\|_{\cL^1,P}<\ve$.}
\]
By Lemma~\ref{approx-stair}(a), such an $\eta$ can be chosen as a staircase
function with values in $\gamma([0,T])\sub S$.\\[2.3mm]
Now assume that~$S$ is convex.\\[2.3mm]
(a) implies (e): For $\gamma$, $\ve$ and $P$ as before,
Lemma~\ref{approx-stair}(b) provides a continuous
function $\theta\colon [0,T]\to \cg$ such that $\|\gamma-\theta\|_{\cL^1,P}<\ve$
(whence $x_0.\evol(\theta)\in U$) and $\theta([0,T])\sub\conv\gamma([0,T])
\sub S$.\\[2.3mm]
(e) implies (c) as each continuous function $[0,T]\to\cg$
is piecewise continuous.\\[2.3mm]
Now assume that~$S$ is convex and $\conv(\ex(S))$
is dense in~$S$.\\[2.3mm]
(d) implies~(f): Let $\gamma\colon [0,T]\to S$ be a staircase
function with $x_0.\evol(\gamma)\in U$.
Let $V$ be an open neighbourhood of $\evol(\gamma)$ in~$G$
such that $x_0.V\sub U$.
There are $0=t_0<t_1<\cdots< t_\ell=T$ such that
$\gamma|_{]t_{j-1},t_j[}$ is a constant function with
value $w_j\in S$ for all $j\in\{1,\ldots, \ell\}$.
The function
\[
\phi\colon \cg^\ell\to \cL^1([0,T],\cg),\;
u=(u_1,\ldots, u_\ell)\mto u_\ell\one_{[t_{\ell-1},t_\ell]}+\sum_{j=1}^{\ell-1} u_j\one_{[t_{j-1},t_j[}
\]
is linear and continuous as $\|\phi(u)\|_{\cL^1,P}\leq T\max\{P(u_1),\ldots,P(u_\ell)\}$
for each continuous seminorm~$P$ on~$\cg$.
Thus there exists an open neighbourhood $W$ of $(w_1,\ldots,w_\ell)$
in $\cg^\ell$ such that $\evol(\phi(W))\sub V$.
Since $(\conv(\ex(S)))^\ell$ is dense in~$S^\ell$ and $(w_1,\ldots, w_\ell)\in S^\ell$,
we find $u=(u_1,\ldots,u_\ell)\in(\conv(\ex(S)))^\ell\cap W$.
After replacing $\gamma$ with $\phi(u)$,
we may assume that $\gamma([0,T])=\{w_1,\ldots,w_\ell\}\sub\conv (\ex(S))$.
Since
\[
\evol(\gamma)=\exp_G((t_1-t_0)w_1)\cdots\exp_G((t_\ell-t_{\ell-1})w_\ell)\in V,
\]
we find open neighbourhoods $W_j$ of $\exp_G((t_j-t_{j-1})w_j)$ in~$G$
for $j\in\{1,\ldots,\ell\}$ such that
\[
W_1W_2\cdots W_\ell\sub V.
\]
For each $j\in \{1,\ldots,\ell\}$,
we have
\[
w_j=\sum_{i=1}^{m_j}t_{j,i}v_{j,i}
\]
for suitable $m_j\in\N$, $v_{j,i}\in \ex(S)$ and $t_{j,i}>0$
for $i\in\{1,\ldots, m_j\}$ with $\sum_{i=1}^{m_j}t_{j,i}=1$.
By Lemma~\ref{Trotter-use},
for each $j\in\{1,\ldots, \ell\}$,
there is a staircase function $\gamma_j\colon [0,t_j-t_{j-1}]\to\cg$
with image in $\{v_{j,1},\ldots, v_{j,m_j}\}\sub\ex(S)$
such~that
\[
\evol(\gamma_j)\in W_j.
\]
Then $\eta:=\gamma_1*\cdots*\gamma_\ell\colon [0,T]\to\cg$
is a staircase function with values in $\ex(S)$ such that
$\evol(\eta)=\evol(\gamma_1)\cdots\evol(\gamma_\ell)
\in W_1\cdots W_\ell\sub V$
and thus $x_0.\evol(\eta)\in U$.\\[2.3mm]
(f) implies (d): This is trivial. $\,\square$.
\section{Continuity of evolution in the {\boldmath$L^1$}-topology}\label{sec-han}
For our next theorem, we need terminology
from~\cite{MeR}.
We recall:
\begin{numba}\label{def-bifun}
Assume that,
for each Fr\'{e}chet space $E$ and real numbers $a<b$,
a vector subspace $\cE([a,b],E)$ of $L^1([a,b],E)$
has been chosen, together with a locally convex vector topology on $\cE([a,b],E)$
such that the inclusion map $\cE([a,b],E)\to L^1([a,b],E)$
is continuous. Following
\cite[Definition 3.1]{MeR}, we call $\cE$ a \emph{bifunctor
on Fr\'{e}chet spaces} if (a) and (b) hold:
\begin{itemize}
\item[(a)]
For each continuous linear map $\lambda\colon E_1\to E_2$
between Fr\'{e}chet spaces,
we have $[\lambda\circ\gamma]\in \cE([a,b],E_2)$
for all $a <b$ and $[\gamma] \in \cE([a,b],E_1)$,
and the linear map $\cE([a,b],\lambda)\colon
\cE([a,b],E_1) \to\cE([a,b],E_2)$, $[\gamma]\mto[\lambda\circ \gamma]$
is continuous.
\item[(b)]
For each Fr\'{e}chet space $E$, real numbers $a<b$ and $c<d$
and each mapping $f \colon [c,d]\to [a,b]$
which is the restriction of a strictly increasing affine-linear map
$\R\to\R$,
we have $[\gamma\circ f]\in\cE([c,d],E)$ for each $[\gamma]\in\cE([a,b],E)$
and the linear map $\cE(f,E)\colon\cE([a,b],E)\to\cE([c,d],E)$,
$[\gamma]\mto[\gamma\circ f]$
is continuous.
\end{itemize}
\end{numba}
\begin{numba}\label{intecompl}
Recall that a locally convex space~$E$
is called \emph{sequentially complete}
(or also: an \emph{sclc-space},
for short) if every Cauchy sequence in~$E$
is convergent in~$E$.
We say that~$E$ is \emph{integral
complete} if the weak integral
$\int_0^1\gamma(t)\,dt$ exists in~$E$
for each continuous function
$\gamma\colon [0,1]\to E$.
It is known (cf.\ \cite{Wei})
that~$E$ is integral complete
if and only if~$E$
has the \emph{metric convex compactness
property} (metric CCP) discussed in~\cite{Voi}, requiring that
the closure of $\conv K$ in~$E$
be compact for each compact, metrizable subset $K\sub E$.
Sequential completeness
implies integral completeness, but not conversely
(cf.\ \cite{Voi}).
The modelling space of every $C^0$-regular
Lie group is integral complete (see \cite[Theorem~C\,(a)]{SEM}).
\end{numba}
\begin{numba}\label{theFEP}
Following \cite[Definition 1.38]{MeR},
we say that a locally convex space~$E$ has the \emph{Fr\'{e}chet exhaustion property}
(FEP) if every separable closed vector subspace $S\sub E$
is a union of vector subspaces $F_1\sub F_2\sub\cdots$
of~$E$ which are Fr\'{e}chet spaces in the induced topology.
All Fr\'{e}chet spaces, strict (LF)-spaces,
and locally convex direct sums of Fr\'{e}chet
spaces are (FEP)-spaces,
as well as their closed vector subspaces.
Bifunctors on sequentially complete (FEP)-spaces
are defined as in \ref{def-bifun}, replacing
Fr\'{e}chet spaces with sequentially complete (FEP)-spaces.
\end{numba}
\begin{numba}\label{theLRC}
If $E$ is a locally convex space, then the vector space $\cL^\infty_{rc}([a,b],E)$
of all Borel measurable functions $\gamma\colon [a,b]\to E$
can be considered such that the closure of $\gamma([a,b])$ in~$E$
is compact and metrizable,
and the corresponding space $L^\infty_{rc}([a,b],E)$
of equivalence classes (see \cite{Mea}).
Bifunctors on integral complete locally convex spaces
are defined as in \ref{def-bifun}, replacing
Fr\'{e}chet spaces with integral complete locally
convex spaces and each symbol $L^1$ with $L^\infty_{rc}$.
\end{numba}
\begin{numba}
Let $\cE$ be a bifunctor on Fr\'{e}chet spaces
$($resp., on sequentially complete
{\rm(FEP)}-spaces, resp.,
on integral complete locally convex spaces$)$
and~$E$ be such a locally convex space.\footnote{For example,
$\cE=L^p$ with $p\in [1,\infty]$ in cases 1 or 2;
or $\cE=L^\infty_{rc}$
in all three cases.}
Given real numbers $a<b$, a function
$\eta\colon [a,b]\to E$ is called a $AC_\cE$
if there exists $[\gamma]\in \cE([a,b],E)$ such that
\begin{equation}\label{absviaint}
\eta(t)=\eta(a)+\int_a^t\gamma(s)\,ds\quad
\mbox{for all $t\in[a,b]$,}
\end{equation}
where the integrals
exist as weak integrals in~$E$, i.e.,
$\lambda(\eta(t))=\lambda(\eta(a))+\int_a^t\lambda(\gamma(s))\,ds$
for all $\lambda\in E'$.
Then $[\gamma]\in\cE([a,b],E)$ is uniquely determined
by~$\eta$ (see \cite{MeR}), and we write $\eta':=[\gamma]$.
\end{numba}
To avoid clumsy formulations, in the following proof
square brackets will frequently be omitted; it will be clear from
the context whether $\gamma$ or $[\gamma]$ is intended.
Notably, we may write $\eta'=\gamma$ for a representative
$\gamma$ of $[\gamma]$.
\begin{thm}\label{L1cont}
Let $\cE$ be a bifunctor on Fr\'{e}chet spaces
$($resp., on sequentially complete
{\rm(FEP)}-spaces, resp.,
on integral complete locally convex spaces$)$
which satisfies the locality axiom, the pushforward
axioms, and such that smooth functions acts smoothly
on $AC_\cE$.
If $\cE$ satisfies the subdivision axiom, then
the evolution map
\[
\Evol\colon \cE([0,1],\cg)\to C([0,1],G)
\]
is continuous with respect to the $L^1$-topology
on $\cE([0,1],\cg)$,
for each Lie group~$G$ modelled
on a locally convex space of the preceding form
such that $G$ is $\cE$-regular.
\end{thm}
{\bf Proof of Theorem~\ref{L1cont}.}
Let $U\sub G$ be an open identity neighbourhood.
After shrinking~$U$, we may assume that there exists
a $C^\infty$-diffeomorphism
$\phi\colon U\to V$ onto an
open $0$-neighbourhood $V\sub \cg$ such that $\phi(e)=0$
and $d\phi|_\cg=\id_\cg$.
Let $q$ be a continuous seminorm on~$\cg$.
Since $G$ is $\cE$-regular, it is $C^0$-regular
(see \cite[Corollaries 5.21 and 5.22]{MeR}) and hence locally
$\mu$-convex (see \cite{Han}).
Thus, we find a continuous seminorm~$p$ on~$\cg$
such that
\begin{equation}\label{herelocm}
g_1\cdots g_n\in U\quad\mbox{and}\quad
q(\phi(g_1\cdots g_n))\leq\sum_{j=1}^np(\phi(g_j))
\end{equation}
for all $n\in \N$ and all $g_1,\ldots, g_n\in U$ such that
$\sum_{j=1}^np(\phi(g_j))< 1$.
Let $Y\sub U$ be an open identity neighbourhood
such that $YY\sub U$. Then $Z:=\phi(Y)\sub V$
is an open $0$-neighbourhood.
The map
\[
\mu\colon Z\times Z\to V,\quad
(x,y)\mto\phi(\phi^{-1}(x)\phi^{-1}(y))
\]
is smooth.
There exist an open $0$-neighbourhood $N\sub Z$
and a continuous seminorm~$P$ on~$\cg$
such that
\[
d_2\mu(N\times \{0\}\times B^P_1(0))\sub B^p_1(0),
\]
entailing that
\[
p(d_2\mu(x,0,y))\leq P(y)\quad\mbox{for all $x\in N$ and $y\in \cg$.}
\]
By continuity of $\Evol\colon\cE([0,1],\cg)\to C([0,1],G)$,
there exists an open $0$-neighbourhood
$W\sub \cE([0,1],\cg)$ such that
\[
\Evol(W)\sub C([0,1],\phi^{-1}(N)).
\]
For $\eta\in W$, we deduce that $\theta:=\phi\circ \Evol(\eta)$
satisfies
\[
\theta(t)=\int_0^t\theta'(s)\,ds
=\int_0^td_2\mu(\theta(s),0,\eta(s))\,ds
\]
for all $t\in [0,1]$, whence
\begin{equation}\label{withgammk}
p(\theta(t))\leq\int_0^tp(d_2\mu(\theta(s),0,\eta(s)))\,ds
\leq\int_0^tP(\eta(s))\,ds
\leq\|\eta\|_{L^1,P}.
\end{equation}
Now let $\gamma\in \cE([0,1],\cg)$
such that $\|\gamma\|_{L^1,P}<1$.
By the subdivision property,
there exists $n\in\N$ such that
$\gamma_{n,k} \in W$ for all $k\in\{0,1,\ldots, n-1\}$,
where
\[
\gamma_{n,k}\colon [0,1]\to\cg,\quad
t\mto \frac{1}{n}\gamma((k+t)/n)
\]
(see \cite[Definition~5.24]{MeR}). We fix~$n$. Note that
\[
\sum_{k=0}^{n-1}\|\gamma_{n,k}\|_{L^1,P}=\|\gamma\|_{L^1,P}<1.
\]
For $k\in\{0,\ldots, n\}$ and $t\in[0,1]$, we get
$\eta_n\!:=\!\Evol(\gamma_{n,k})\!\in\! C([0,1],\phi^{-1}(N))$~and
\[
p(\phi(\eta_n(t)))\leq\|\gamma_{n,k}\|_{L^1,P},
\]
by (\ref{withgammk}).
Given $t\in [0,1]$, there is $k\in\{0,\ldots,n-1\}$
with $t\in [k/n,(k+1)/n]$.
By (\ref{herelocm}), we have
\[
\Evol(\gamma)(t)=\evol(\gamma_{n,0})\cdots\evol(\gamma_{n,k-1})\Evol(\gamma_{n,k})(nt-k)
\in U,
\]
since $g_j:=\evol(\gamma_{n,j-1})$ for $j\in\{1,\ldots,k\}$
and $g_{k+1}:=\Evol(\gamma_{n,k})(nt-k)$ are elements of~$\phi^{-1}(N)\sub U$
which satisfy
\[
\sum_{j=1}^{k+1}p(\phi(g_j))\leq\sum_{j=0}^k\|\gamma_{n,j}\|_{L^1,P}
<1.
\]
Thus $\Evol(\gamma)\in C([0,1],U)$ for all $\gamma\in\cE([0,1],\cg)$
such that $\|\gamma\|_{L^1,P}<1$, showing that $\Evol\colon\cE([0,1],\cg)\to C([0,1],G)$
is continuous at~$0$ with respect to the $L^1$-topology.\\[2.3mm]
Since $G$ is $\cE$-regular,
$\cE([0,1],\cg)$ can be made a group with neutral element~$0$
and group multiplication given by
\[
[\gamma]\odot[\eta]:=[\Ad(\Evol([\eta]))^{-1}.\gamma]+[\eta]
\]
(see \cite[Definition~5.34]{MeR}).
In view of~\cite[Lemma~5.10]{MeR},
the right translation
\[
\rho_{[\eta]}\colon \cE([0,1],\cg)\to\cE([0,1],\cg),\quad
[\gamma]\mto [\gamma]\odot [\eta]
\]
is continuous with respect to the $L^1$-topology on
both sides
%
% more details or references to help readers?
%
and hence a homeomorphism,
for each $[\eta]\in\cE([0,1],\cg)$.
For $\beta:=\Evol([\eta])$, the right translation
\[
\rho_\beta\colon C([0,1],G)\to C([0,1],G),\quad
\zeta\mto \zeta\beta
\]
is continuous. Since $\Evol$ is continuous at~$0$, we deduce from
\[
\Evol=\rho_\beta\circ \Evol\circ \, \rho_{[\eta]^{-1}}
\]
that the map $\Evol\colon \cE([0,1],\cg)\to C([0,1],G)$
is continuous at~$[\eta]$ with respect to the $L^1$-topology
on its domain.$\,\square$
\begin{rem}\label{basics-pw-evol}
Let $G$ be a $C^0$-regular Lie group modelled on a locally
convex space
and $PC([0,1],\cg)$ be the space of piecewise continuous
$\cg$-valued functions on~$[0,1]$.
Given $\gamma\in PC([0,1],\cg)$,
let $0=t_0<\cdots<t_m=1$ be a subdivision of
$[0,1]$ such that $\gamma|_{]t_{j-1},t_j[}$
has a continuous extension $\gamma_j\colon [t_{j-1},t_j]\to\cg$
for all $j\in\{1,\ldots, m\}$.
Given $t\in [0,1]$, there is $j\in\{1,\ldots,m\}$
such that $t\in [t_{j-1},t_j]$.
We define
\[
\Evol(\gamma)(t):=\evol(\gamma_1)\cdots\evol(\gamma_{j-1})\Evol(\gamma_j)(t),
\]
using evolution maps $\evol$ with domain $C([t_{i-1},t_i],\cg)$
for $i\in\{1,\ldots,j-1\}$ and $\Evol$
with domain $C([t_{j-1},t_j],\cg)$
on the right-hand side.
Then $\Evol(\gamma)\in C([0,1],G)$ is well defined, independent
of the choice of subdivision,\footnote{It is unchanged if we
add one point to a subdivision, and hence under passage to joint
refinements of two given subdivisions.} and we
obtain a map
\begin{equation}\label{pwevol}
\Evol\colon PC([0,1],\cg)\to C([0,1],G).
\end{equation}
It is known from the work of Hanusch that
the evolution map (\ref{pwevol})
is continuous with respect to the $\cL^1$-topology\footnote{By definition, this is the
(not necessarily Hausdorff)
locally convex vector topology defined
by the seminorms $PC([0,1],\cg)\to[0,\infty[$,
$\gamma\mto\|\gamma\|_{\cL^1,q}$
for $q$ in the set of all continuous
seminorms on $\cg$.}
on its domain $PC([0,1],\cg)$.
This can be shown as in the previous proof,
with the modification that the subdivisions used to define
the $\gamma_{n,k}$ need not be equidistant but include the
points $t_0<\cdots<t_m$ just considered (and hence
the points of discontinuity of~$\gamma$), to ensure that
the $\gamma_{n,k}$ are continuous functions;
moreover, the subdivisions need to be chosen such that
the mesh tends to $0$ for $n\to\infty$.
\end{rem}
Using an affine-linear reparametrization,
we can replace $[0,1]$ with $[a,b]$ for any $a<b$
in the preceding results concerning evolution maps
(cf.\ \ref{repara}).\\[2.3mm]
{\bf Proof of Theorem~\ref{thmD}.}
Theorem~\ref{thmD}
becomes a special case of Theorem~\ref{L1cont}
if we consider $L^p$ as a bifunctor on
Fr\'{e}chet spaces. The required axioms and hypotheses
were verified in \cite{MeR}. $\,\square$
\section{Analogues using only {\boldmath$L^q$}-regularity}\label{sec-Lq}
{\bf Proof of Theorem~\ref{thmE}.}
(b) implies (a)$'$ since $\cL^p([0,T],\cg)\sub \cL^q([0,T],\cg)$.\\[2.3mm]
The implications ``(d)$\impl$(c)'', ``(c)$\impl$(b)'',
``(e)$\impl$(c)'' and ``(f)$\impl$(d)'' can be shown as in the proof
of Theorem~\ref{thmC}.\\[2.3mm]
(a)$'$ implies (d):
Let $\gamma\in \cL^q([0,T],\cg)$
such that $x_0.\evol(\gamma)\in U$.
Since $\evol\colon \cL^q([0,T],\cg)\to G$
is continuous with respect to the $\cL^1$-topology
(cf.\ Theorem~\ref{thmD})
and the action $M\times G\to M$ is continuous,
there exists $\ve>0$ and a continuous seminorm~$P$ on
$\cg$ such that
\[
x_0.\evol(\eta)\in U \mbox{ for all $\eta\in \cL^q([0,T],\cg)$ such that
$\|\gamma-\eta\|_{\cL^1,P}<\ve$.}
\]
By Lemma~\ref{approx-stair}(a), such an $\eta$ can be chosen as a staircase
function with values in $\gamma([0,T])\sub S$.\\[2.3mm]
The implication ``(a)$'${}$\impl$(e)''
can be shown like  ``(a)$\impl$(e)'' in the proof of Theorem~\ref{thmC}.\\[2.3mm]
The implication ``(d)$\impl$(f)''
can be shown as in the proof of Theorem~\ref{thmC},
except that we consider
$\phi$ as a map from $\cg^\ell$ to $\cL^q([0,T],\cg)$,
endowed with the $\cL^1$-topology. $\,\square$
\section{Analogues requiring only {\boldmath$C^0$}-regularity}\label{sec-C0}
For $C^0$-regular Lie groups,
we can say the following.
\begin{thm}\label{thmF}
Instead of requiring $L^1$-regularity,
let $G$ be a Lie group modelled
on a locally convex space
such that $G$ is $C^0$-regular
in the situation of Theorem~{\rm\ref{thmC}}.
Then conditions~{\rm(c)} and {\rm(d)} of
Theorem~{\rm\ref{thmC}} are equivalent.
If~$S$ is convex, then {\rm(d)} is equivalent to {\rm(e)}.
If~$S$ is convex and $\conv(\ex(S))$ is dense in~$S$,
then {\rm(d)} is equivalent to~{\rm(f)}.
\end{thm}
\begin{proof}
The implications ``(d)$\impl$(c)'', ``(e)$\impl$(c)'', and
``(f)$\impl$(d)'' are trivial.\\[2.3mm]
(c) implies (d):
Let $\gamma\in PC([0,T],\cg)$
such that $x_0.\evol(\gamma)\in U$.
Since $\evol\colon PC([0,T],\cg)\to G$
is continuous with respect to the $\cL^1$-topology
(cf.\ Remark~\ref{basics-pw-evol})
and the action $M\times G\to M$ is continuous,
there exists $\ve>0$ and a continuous seminorm~$P$ on
$\cg$ such that
\[
x_0.\evol(\eta)\in U \mbox{ for all $\eta\in PC([0,T],\cg)$ such that
$\|\gamma-\eta\|_{\cL^1,P}<\ve$.}
\]
By the conclusion of Lemma~\ref{approx-stair}(a), such an $\eta$ can be chosen as a staircase
function with values in $\gamma([0,T])\sub S$ (see Remark~\ref{appr-lcx}).\\[2.3mm]
(c) implies (e): In view of Remark~\ref{appr-lcx},
starting with a piecewise continuous function
$\gamma$, this can be shown like  ``(a)$\impl$(e)'' in the proof
of Theorem~\ref{thmC}.\\[2.3mm]
The implication ``(d)$\impl$(f)''
can be shown as in the proof of Theorem~\ref{thmC},
except that we consider
$\phi$ as a map from $\cg^\ell$ to $PC([0,T],\cg)$,
endowed with the $\cL^1$-topology.
\end{proof}
\section{Analogues beyond Fr\'{e}chet-manifolds}\label{sec-beyond}
In this section, we explain how our results can be extended
from the case of Fr\'{e}chet manifolds (and Fr\'{e}chet-Lie groups)
to the case of manifolds (and Lie groups) modelled on sequentially
complete locally convex spaces.
The main point is that an analogue of Lemma~\ref{locflow-uniq}
can be
established also in the current higher generality
(see Lemma~\ref{secompl-flo}),
as well as an analogue of Theorem~\ref{thmB}
(see Theorem~\ref{thmBgen}).
Once this foundation is established in the necessary detail,
it will be enough to revisit the other results,
and describe which minor modifications are necessary
in the statements and proofs.
\begin{numba}\label{def-lus}
We recall that a mapping $\gamma\colon I\to X$
from an interval $I\sub \R$ to a topological space~$X$
is called \emph{Lusin measurable}
if there exists a sequence $(K_j)_{j\in\N}$
of compact subsets $K_j\sub I$ such that
\begin{itemize}
\item[(i)]
The restriction $\gamma|_{K_j}\colon K_j\to X$ is
continuous for each $j\in \N$;
\item[(ii)]
$\lambda_1(I\setminus \bigcup_{j\in \N}K_j)=0$.
\end{itemize}
See \cite{FMP} and the references therein for
further information, also \cite{Nik}.
\end{numba}
\begin{rem}\label{close-to-borel}
(a)
Let $\tilde{\lambda}\colon \tilde{\cB}(\R)\to[0,\infty]$ be Lebesgue
measure. If $X$ is second countable,
then a map $\gamma\colon I\to X$ is Lusin measurable
if and only if $\gamma$ is measurable as a
function
from $(I,\tilde{\cB}(I))$ to $(X,\cB(X))$,
where $\tilde{\cB}(I):=\tilde{\cB}(\R)|_I$
(see \cite[Lemma~4.1.8]{Nik}).\medskip

\noindent(b)
If a Lusin measurable map $\gamma\colon I\to X$
and $(K_j)_{j\in\N}$
are as in \ref{def-lus},
after replacing $\gamma(x)$ with a constant
$c\in X$ for all $x\in I\setminus \bigcup_{j\in \N}K_j=:K_0$,
we can achieve that $\gamma$ is
measurable from $(I,\cB(I))$ to $(X,\cB(X))$.
If $X$ is Hausdorff, moreover $\gamma(K_j)$
is compact and metrizable and hence separable and metrizable
for each $j\in\N$ (and so is $\gamma(K_0)$, which is
$\{c\}$ or $\emptyset$).
Notably, $\gamma(I)=\bigcup_{j\in \N_0}\gamma(K_j)$ is separable.
\end{rem}
\begin{numba}
Henceforth,
if $E$ is a locally convex space,
$I\sub \R$ an interval
and $p\in [1,\infty]$,
we write $\cL^p(I,E)$
for the vector space
of all Lusin measurable mappings $\gamma\colon I\to E$
such that $\|\gamma\|_{\cL^p,q}:=\|q\circ\gamma\|_{\cL^p}<\infty$
for all continuous seminorms $q$ on~$E$.
We give $\cL^p(I,E)$ the locally convex vector topology
defined by the seminorms $\|\cdot\|_{\cL^p,q}$.
In the usual way, one now obtains Hausdorff
locally convex spaces $L^p(I,E)$
of equivalence classes $[\gamma]$
modulo Lusin  measurable functions vanishing
almost everywhere
(see \cite[p.\,43]{Nik}).
If $\gamma\colon I\to E$ is \emph{locally integrable}
in the sense that $\gamma|_{[a,b]}\in\cL^1([a,b],E)$
for all $a<b$ with $[a,b]\sub I$,
again we write $[\gamma]$ for the equivalence
class modulo Lusin measurable functions
which vanish almost everywhere.
\end{numba}
\begin{numba}
If $E$ is an sclc-space,
given real numbers $a<b$, following \cite{Nik}
we call a function $\eta\colon [a,b]\to E$
\emph{absolutely continuous} if
it is a primitive of some $\gamma\in\cL^1([a,b],E)$,
as in~(\ref{absviaint}).
Then $\eta':=[\gamma]$ is uniquely determined (see \cite[Lemma~4.2.6]{Nik}).
Let $p\in [1,\infty]$.
If $\eta'=[\gamma]$ with $\gamma\in \cL^p([a,b],E)$,
we call~$\eta$ an \emph{$AC_{L^p}$-map}.
If $I\sub\R$ is an interval,
we say that a function $\eta\colon I\to E$
is absolutely continuous (resp., $AC_{L^p}$)
if $\eta|_{[a,b]}$ is so for all real numbers
$a<b$ such that $[a,b]\sub I$.
Again, there is a locally
integrable function $\gamma\colon I\to E$
with primitive~$\eta$ and $\eta':=[\gamma]$ is uniquely determined.
\end{numba}
\begin{numba}\label{charu}(Chain Rule).
Let $E$ and $F$ be sclc-spaces,
$U\sub E$ be an open subset, $f\colon  U\to F$ a $C^1$-map
and $\eta\colon I\to E$ be an absolutely continuous function
on a non-degenerate interval $I\sub \R$ such that $\eta(I)\sub U$.
Let $\gamma\colon I\to E$ be a locally integrable function with $\eta'=[\gamma]$.
Then $f\circ\eta\colon I\to F$ is absolutely
continuous and
\[
(f\circ\eta)'=[t\mto df(\eta(t),\gamma(t))],
\]
by \cite[Lemma~4.2.16]{Nik} and its proof.
\end{numba}
\begin{numba}\label{numba-ode}
If $E$ is an sclc-space, $W\sub\R\times E$ a subset,
$f\colon W\to E$ a function
and $(t_0,y_0)\in W$,
we call a function $\gamma\colon I\to E$ on a non-degenerate interval $I\sub \R$
a \emph{Carath\'{e}odory solution} to the initial value problem (\ref{the-ivp})
if $\gamma$ is absolutely continuous, $t_0\in I$ holds, $
(t,\gamma(t))\in W$ for all $t\in I$, and the integral
equation~(\ref{cara2})
is satisfied, or equivalently
\begin{equation}\label{cara3}
\gamma'=[t\mto f(t,\gamma(t))]\quad\mbox{and}\quad
\gamma(t_0)=y_0.
\end{equation}
\end{numba}
\begin{numba}\label{substitute}
If $q\colon E\to F$ is a continuous linear map to a Fr\'{e}chet
space in the situation of~\ref{numba-ode},
then $q\circ \gamma$ is absolutely continuous
and
\[
(q\circ \gamma)'=[t\mto q(f(t,\gamma(t)))],
\]
by \ref{charu}. As $q\circ \gamma$ is an absolutely
continuous map to a Fr\'{e}chet space,
we deduce that there is a Borel set $I_0\sub I$
with $\lambda_1(I\setminus I_0)=0$
such that $q\circ \gamma$ is differentiable at each $t\in I_0$ and
\[
(q\circ\gamma)'(t)=q(f(t,\gamma(t)))\;\; \mbox{for all $t\in I_0$.}
\]
\end{numba}
\begin{numba}
Let $M$ be a $C^1$-manifold modelled on a locally
convex space and $TM$ be its tangent bundle, with the bundle
projection $\pi_{TM}\colon TM\to M$.
If $\gamma\colon I\to TM$ is a Lusin measurable
function on an interval $I\sub\R$, we write $[\gamma]$
for the set of all Lusin measurable functions
$\eta\colon I\to TM$
such that $\pi_{TM}\circ\gamma=\pi_{TM}\circ\eta$
and $\gamma(t)=\eta(t)$ for almost all $t\in I$.
\end{numba}
\begin{numba}
Let $M$ be a $C^1$-manifold modelled on an sclc-space~$E$.
For real numbers $a<b$, consider a continuous function
$\eta\colon [a,b]\to M$.
If $\eta([a,b])\sub U_\phi$ for some chart $\phi\colon U_\phi\to V_\phi\sub E$,
we say that $\eta$ is \emph{absolutely continuous}
if $\phi\circ \eta\colon I\to E$ is so,
and let
\[
\dot{\eta}:=[t\mto T\phi^{-1}((\phi\circ\eta)(t),\gamma(t))]
\]
with $\gamma\in\cL^1([a,b],E)$ such that $(\phi\circ\eta)'=[\gamma]$.
By \ref{charu}, absolute continuity of~$\eta$
is independent of the choice of~$\phi$,
and so is~$\dot{\eta}$.
In the general case,
we call $\eta$ absolutely continuous
if $[a,b]$ can be subdivided
into subintervals $[t_{j-1},t_j]$
such that $\eta([t_{j-1},t_j])$
is contained in a chart domain
and $\eta|_{[t_{j-1},t_j]}$ is absolutely
continuous. If $(\eta|_{[t_{j-1},t_j]})^{\!\textstyle\cdot}=[\gamma_j]$,
we let $\dot{\eta}:=[\gamma]$ with $\gamma(t):=\gamma_j(t)$
if $t\in [t_{j-1},t_j[$ or $j$ is maximal and $t\in [t_{j-1}, t_j]$.
If $I\sub\R$ is an interval, we call a function $\eta\colon I\to M$
absolutely continuous if $\eta|_{[a,b]}$ is so for all $a<b$ such that
$[a,b]\sub I$. We define $\dot{\eta}=[\gamma]$ where
$\gamma$ is defined piecewise using representatives
of $(\eta|_{[a,b]})^{\!\textstyle\cdot}$
for $[a,b]$ in a countable cover of~$I$.
\end{numba}
\begin{numba}\label{chainRNEW}
Let $f\colon M\to N$ be a $C^1$-map between
$C^1$-manifolds modelled on sclc-spaces.
Let $I\sub \R$ be a non-degenerate interval and $\eta\colon I\to M$
be absolutely continuous. Let $\gamma\colon I\to TM$
be a Lusin measurable function such that $\pi_{TM}\circ\gamma=\eta$
and $\dot{\eta}=[\gamma]$.
Then $f\circ\eta\colon I\to N$ is absolutely
continuous and
\[
(f\circ \eta)^{\!\textstyle\cdot}=[t\mto Tf(\gamma(t))],
\]
as a consequence of~\ref{charu}.
\end{numba}
\begin{numba}\label{def-ode-mfd-2}
If $M$ is a $C^1$-manifold modelled on an sclc-space, $W\sub\R\times M$ a subset,
$f\colon W\to TM$ a function such that $f(t,y)\in T_yM$
for all $(t,y)\in W$ and $(t_0,y_0)\in W$,
we call a function $\gamma\colon I\to M$ on a non-degenerate interval $I\sub \R$
a \emph{Carath\'{e}odory solution} to the initial value problem
\begin{equation}\label{ivp-xyz}
\dot{y}(t)=f(t,y(t)),\quad y(t_0)=y_0
\end{equation}
if $\gamma$ is absolutely continuous, $t_0\in I$ holds,
$(t,\gamma(t))\in W$ for all $t\in I$,
\begin{equation}\label{cara13}
\dot{\gamma}\, =\, [t\mto f(t,\gamma(t))],\quad\mbox{and}\quad
\gamma(t_0)=y_0.
\end{equation}
Solutions to the differential equation $\dot{y}(t)=f(t,y(t))$
are defined analogously.
\end{numba}
\begin{numba}
Using the concepts of absolutely continuous functions, $AC_{L^p}$-functions,
and Carath\'{e}odory solutions just described,
one can define $L^p$-regular Lie groups
modelled on sclc-spaces
for $p\in [1,\infty]$ (see \cite[Definition~4.3.7]{Nik}),
as in the case of Fr\'{e}chet-Lie groups
already discussed in this article.
\end{numba}
\begin{rem}
Although $\eta'=[\gamma]$
is still uniquely determined
for an absolutely continuous function
$\eta\colon [a,b]\to E$ to an sclc-space,
in contrast to the Fr\'{e}chet case
$\eta'(t)$ may not exist for almost
all $t$, so that we cannot work
with genuine \emph{derivatives at a point}
anymore.
Generalizing the results and proofs
obtained so far in the Fr\'{e}chet case,
this problem can be frequently be avoided
by replacing statements like
\[
\eta'(t)=\gamma(t)\;\,\mbox{for almost all $t$}
\]
with $\eta'=[t\mto\gamma(t)]$.
Notably, Carath\'{e}odory solutions
to initial value problems
are now defined via (\ref{cara3})~and~(\ref{cara13}).\\[2.3mm]
For example, using this interpretation,
Definitions~\ref{defn-loc-uni} and \ref{defn-loc-ex}
can be extended
to the case that $E$ is an sclc-space,
and Lemma~\ref{la-loc-glob}
and its proof remain valid in this generality.\\[2.3mm]
There is one exception, however:
Derivatives at a point played an
essential role in the proof
of Lemma~\ref{locflow-uniq}.
We therefore discuss the necessary adaptations
in detail now, starting with a suitable modification
of Definition~\ref{defn-loc-flo}.
As we shall see, the property~\ref{substitute}
will be good enough to achieve our goals.
\end{rem}
\begin{defn}\label{defn-loc-flo-2}
Let $J\sub \R$ be a non-degenerate interval, $E$ be an
sclc-space,
$U\sub E$ be a subset
and $f\colon W\to E$ be a function on an open subset $W\sub J\times U$.
Let $k\in\N\cup\{\infty\}$.
We say that the differential equation $y'(t)=f(t,y(t))$
\emph{admits local flows which are pullbacks of $C^k$-maps}
if,
for all $(\bar{t},\bar{y}) \in W$,
there exist a compact interval $I\sub J$ which is a neighbourhood
of~$\bar{t}$ in~$J$,
an open neighborhood~$V$ of $\bar{y}$ in~$U$
with $I\times V\sub W$
and function
\[
\Phi\colon I\times I\times V\to E
\]
such that conditions {\rm(a)} and {\rm(b)}
from Definition~{\rm\ref{defn-loc-flo}}
are satisfied, and~{\rm(c)$'$}:
\begin{itemize}
\item[{\rm(c)}$'$]
There exist sclc-spaces~$E_1$ and $E_2$,
open subsets $V_1\sub E_1$ and $V_2\sub E_2$,
absolutely continuous functions
$\alpha\colon I\to V_1\sub E_1$ and $\beta\colon I\to V_2\sub E_2$,
and a $C^k$-map $\Psi\colon V_1\times V_2 \times V\to E$
such that
\[
\Phi(t,t_0,y_0)=\Psi(\alpha(t),\beta(t_0),y_0)\quad\mbox{for all $(t,t_0,y_0)
\in I\times I\times V$.}
\]
Moreover, for each continuous linear map $q\colon E\to F$
to a Banach space~$F$ and each compact subset $C\sub V$,
we require the existence of a Borel set $I_q\sub I$
with $\lambda_1(I\setminus I_q)=0$
and an open subset $\Omega_q\sub V$ with $C\sub \Omega_q$
such that
$\frac{d}{dt}q(\Phi_{t,t_0}(y_0))$
exists
and
\[
\frac{d}{dt}q(\Phi_{t,t_0}(y_0))=q(f(t,\Phi_{t,t_0}(y_0)))
\]
for all $y_0\in \Omega_q$, $t_0\in I$, and $t \in I_q$.
\end{itemize}
\end{defn}
Note that Remark~\ref{flo-then-ex}
remains valid in sclc-spaces if we replace Definition~\ref{defn-loc-flo}
with Definition~\ref{defn-loc-flo-2}.
\begin{la}\label{secompl-flo}
Let $J\sub \R$ be a non-degenerate interval, $E$ be an
sclc-space,
$U\sub E$ be a subset
and $f\colon W\to E$ be a function on an open subset $W\sub J\times U$.
If the differential equation $y'(t)=f(t,y(t))$
admits local flows which are pullbacks of $C^2$-maps,
then it satisfies local uniqueness
of Carath\'{e}odory solutions.
\end{la}
\begin{proof}
Let $\gamma_j\colon I_j\to E$
be solutions to $y'(t)=f(t,y(t))$
for $j\in \{1,2\}$ and $\bar{t}\in I_1\cap I_2$ such that
$\bar{y}:=\gamma_1(\bar{t})=\gamma_2(\bar{t})$.
To see that $\gamma_1$ and $\gamma_2$ coincide on a neighbourhood of~$\bar{t}$
in $I_1\cap I_2$, we may assume
that $I_1\cap I_2$ is a non-degenerate interval.
Let $I$, $V$, $\Phi$, $Y$, $E_1$, $E_2$, $V_1$, $V_2$,
$\Psi$, $\alpha$, and $\beta$
be as in Definition~\ref{defn-loc-flo-2}.\\[2.3mm]
There exists a compact interval $K\sub I_1\cap I_2\cap I$
which is a neighbourhood of $\bar{t}$ in $I_1\cap I_2\cap I$
such that $\gamma_1(K)\sub Y$,
$\gamma_2(K)\sub Y$ and $\Phi_{t,\bar{t}}(\bar{y})\in Y$ for all $t\in K$.
After shrinking~$K$ if necessary, we can also assume that
\[
\theta_j(t):=\Phi_{\bar{t},t}(\gamma_j(t))\in Y \;\;\mbox{for all
$t\in K$ and $j\in\{1,2\}$.}
\]
Note that $\theta_j$ is absolutely continuous
by \ref{charu} as
$\theta_j(t)=\Psi(\alpha(\bar{t}),\beta(t),\gamma_j(t))$.
It suffices to show that
\[
\gamma_j(t)=\Phi_{t,\bar{t}}(\bar{y})\;\;\mbox{for
all $t\in K$ and $j\in\{1,2\}$.}
\] 
Since $\Phi_{t,\bar{t}}\circ \Phi_{\bar{t},t}|_Y=\id_Y$ for all
$t\in I$ (by (a) and (b) in Definition~\ref{defn-loc-flo-2}), the map
$\Phi_{\bar{t},t}|_Y$ is injective.
Hence $\gamma_1|_K=\gamma_2|_K$
will hold if we can show that both $\theta_1$ and $\theta_2$ coincide
with
\[
\theta\colon K\to E,\quad t\mto \Phi_{\bar{t},t}(\Phi_{t,\bar{t}}(\bar{y}))=\bar{y}.
\]
It suffices to show that
\[
q\circ\theta=q\circ \theta_j
\]
for $j\in\{1,2\}$, for each continuous
linear map $q\colon E\to \R$.
Since $\theta_j(\bar{t})=\bar{y}=\theta(\bar{t})$ for $j\in\{1,2\}$,
the latter will hold if we can show that
\[
(q\circ \theta_j)'(t)=(q\circ \theta)'(t)=0
\]
for almost all $t\in K$.
By Lemma~\ref{factorCk},
there exist Banach spaces $F_j$ for $j\in\{1,2\}$ and $F$,
open subsets $Q_j\sub F_j$ and $Q\sub F$,
continuous linear maps $p_j\colon E_j\to F_j$
and $p\colon E\to F$,
and open subsets $V_{j,0}\sub V_j$ and $V_0\sub V$,
such that
\[
\alpha(K)\sub V_{1,0},\quad \beta(K)\sub V_{2,0},\quad
\gamma_1(K)\cup\gamma_2(K)\sub V_0,
\]
\[
p_1(V_{1,0})\sub Q_1,\quad p_2(V_{2,0})\sub Q_2,\quad
p(V_0)\sub Q
\]
and a $C^1$-map $\Psi_q\colon Q_1\times Q_2\times Q\to \R$
such that
\[
q(\Psi(x,y,z))=\Psi_q(p_1(x),p_2(y),p(z))\quad
\mbox{for all $\,(x,y,z)\in V_{1,0}\times V_{2,0}\times V_0$.}
\]
For $j\in \{1,2\}$, there exists a Borel set $I_{j,0}\sub I_j$
with $\lambda_1(I_j\setminus I_{j,0})=0$
such that $p\circ \gamma_j$ is differentiable at each $t\in I_{j,0}$
and
\[
(p\circ \gamma_j)'(t)\, =\, p(f(t,\gamma_j(t)))
\mbox{ \,for all $\, t\in I_{j,0}$,}
\]
see~\ref{substitute}.
Let~$I_p\sub I$ and $\Omega_p$ be analogous
to~$I_q$ and $\Omega_q$ in Definition~\ref{defn-loc-flo-2}\,(c)$'$,
applied with~$p$ in place of~$q$
and $C:=\theta_1(K)\cup\theta_2(K)$.
After shrinking~$I_p$,
we may assume that, moreover, $(p_2\circ \beta)'(t)$
exists at each $t\in I_p$.
For each $t\in K$, we have
\begin{eqnarray*}
q(\theta_j(t))&=&q(\Phi(\bar{t},t,\gamma_j(t)))
=q(\Psi(\alpha(\bar{t}),\beta(t),\gamma_j(t)))\\
&=&\Psi_q((p_1\circ\alpha)(\bar{t}),
(p_2\circ\beta)(t),(p\circ\gamma_j)(t)).
\end{eqnarray*}
Differentiating at $t\in K_0:= K \cap I_{1,0}\cap I_{2,0}\cap I_p$,
we obtain
\begin{eqnarray*}
\frac{d}{dt}\,
q(\theta_j(t)) &=&
\frac{d}{dt}\,
\Psi_q((p_1\circ\alpha)(\bar{t}),(p_2\circ\beta)(t),(p\circ\gamma_j)(t)))\\
&=&
d_2\Psi_q((p_1\circ\alpha)(\bar{t}),(p_2\circ\beta)(t),(p\circ\gamma_j)(t);
(p_2\circ\beta)'(t))\\
& & \;+
d_3\Psi_q((p_1\circ\alpha)(\bar{t}),(p_2\circ\beta)(t),
(p\circ\gamma_j)(t);\underbrace{(p\circ\gamma_j)'(t)}_{=p(f(t,\gamma_j(t)))}).
\end{eqnarray*}
Using the Chain Rule, the second summand
can be rewritten as
\[
d_3(q\circ\Psi)(\alpha(\bar{t}),\beta(t),
\gamma_j(t);f(t,\gamma_j(t)))
\,=\,
q(d\Phi_{\bar{t},t}(\gamma_j(t),f(t,\gamma_j(t)))).
\]
Assume that $z\in Y\cap\Omega_p$ and $t\in K_0$
are given such that $\Phi(t,\bar{t},z)\in V_0$.
Then $\Phi(\tau,\bar{t},z)\in V_0$
for $\tau\in K$ close to~$t$, and
\begin{equation}\label{indep}
q(z)=q(\Phi(\bar{t},\tau,\Phi_{\tau,\bar{t}}(z)))
=\Psi_q((p_1\circ\alpha)(\bar{t}),(p_2\circ\beta)(\tau),p(\Phi_{\tau,\bar{t}}(z)))
\end{equation}
is independent of~$\tau$. Differentiating (\ref{indep})
with respect to~$\tau$
at $\tau=t$, we obtain
\begin{eqnarray}
0 &=&
d_2\Psi_q((p_1\circ\alpha)(\bar{t}),(p_2\circ\beta)(t),p(\Phi_{t,\bar{t}}(z));
(p_2\circ\beta)'(t)) \notag\\
& & \;
+\,q(d\Phi_{\bar{t},t}(\Phi_{t,\bar{t}}(z),f(t,\Phi_{t,\bar{t}}(z)))),\label{Git}
\end{eqnarray}
repeating the arguments used to calculate
$\frac{d}{dt}\,q(\theta_j(t))$; note that
\[
\frac{d}{dt}p(\Phi_{t,\bar{t}}(z))=p(f(t,\Phi_{t,\bar{t}}(z)))
\]
since $t\in I_p$ and $z\in\Omega_p$.
For $j\in \{1,2\}$ and any $t\in K_0$,
we may choose $z:=\theta_j(t)$
here. In fact, then $z\in\Omega_p$ by choice of~$C$;
also, $z=\theta_j(t)\in Y$
and
\[
\Phi_{t,\bar{t}}(z)=\Phi_{t,\bar{t}}(\Phi_{\bar{t},t}(\gamma_j(t)))
=\gamma_j(t)\in V_0.
\]
Substituting $\Phi_{t,\bar{t}}(z)=\gamma_j(t)$ into (\ref{Git}),
we obtain
\begin{eqnarray*}
0 &=&
d_2\Psi_q((p_1\circ\alpha)(\bar{t}),(p_2\circ\beta)(t),(p\circ \gamma_j)(t);
(p_2\circ\beta)'(t))\\
& & \;\,
+ \, q(d\Phi_{\bar{t},t}(\gamma_j(t),f(t,\gamma_j(t))))
\; =\;  \frac{d}{dt}\,q(\theta_j(t)),
\end{eqnarray*}
by the above calculation.
\end{proof}
\begin{rem}\label{all-holds}
Remark~\ref{mfdlocalize}
concerning associated differential equations in local charts,
Definition~\ref{def-locuni-mfd} of local uniqueness,
and Lemma~\ref{mfd-loc-glob} (and its proof) concerning global uniqueness
remain meaningful for manifolds with
sequentially complete modelling spaces, if Definition~\ref{def-ode-mfd}
is replaced with \ref{def-ode-mfd-2}.
In Definition~\ref{def-locex-mfd}
of local existence, Lemma~\ref{max-sol}
concerning maximal solutions, and Definition~\ref{max-flow}
of maximal flows, we can replace Fr\'{e}chet spaces with
sclc-spaces,
without further changes.
\end{rem}
\begin{defn}\label{flopbnew}
Let $M$ be a $C^k$-manifold
modelled on an sclc-space, with $k\in\N\cup\{\infty\}$.
Let $J\sub\R$ be a non-degenerate interval,
$W\sub J\times M$ be an open subset,
and $f\colon W\to TM$ be a function such that
$f(t,y)\in T_yM$ for all $(t,y)\in W$.
We say that the differential equation
$\dot{y}(t)=f(t,y(t))$
\emph{admits local flows which are pullbacks of $C^k$-maps}
if $y'(t)=f_\phi(t,y(t))$ does
so, for each chart $\phi\colon U_\phi\to V_\phi$.
\end{defn}
\begin{rem}\label{thmAgen}
(a) Using Definition~\ref{flopbnew},
we immediately deduce from Lemma~\ref{secompl-flo}
that Proposition~\ref{thmA}
remains valid for differential
equations admitting local flows
which are pullbacks of $C^2$-maps,
on a $C^2$-manifold $M$ modelled
on an sclc-space (instead of
a Fr\'{e}chet space).\medskip

\noindent
(b) In the case of a Fr\'{e}chet manifold,
the condition in Definition~\ref{def-loc-flo-mfd}
implies the one in Definition~\ref{flopbnew}
(see proof of Proposition~\ref{thmA}).
We avoided a direct analogue
of Definition~\ref{def-loc-flo-mfd} 
in the case of sequentially complete
modelling spaces.
The proof of Theorem~\ref{thmB}
now becomes more technical.
\end{rem}
\begin{thm}\label{thmBgen}
Theorem~{\rm\ref{thmB}}
remains valid if $G$ and $M$ are modelled
on sequentially complete locally convex spaces
$($which need not be Fr\'{e}chet spaces$)$.
\end{thm}
\begin{proof}
Given $g\in G$, write $\lambda_g\colon G\to G$, $h\mto gh$
for left translation by~$g$.
Write $\sigma_y\colon G\to M$, $g\mto\sigma(y,g)$ for $y\in M$;
thus
\begin{equation}\label{sta0}
T\sigma_y(v)=v_\sharp(y)\quad\mbox{for all $\,v\in\cg$.}
\end{equation}
Since $\sigma(y,gh)=\sigma(\sigma(y,g),h)$ for all $y\in M$ and $g,h\in G$, we have
$\sigma_y\circ\lambda_g=\sigma_{\sigma(y,g)}$ and thus
\begin{equation}\label{sta1}
T\sigma_y\circ T\lambda_g=T\sigma_{\sigma(y,g)}\quad \mbox{for all $\,y\in M$
and $g\in G$.}
\end{equation}
If we define $\Fl$ as in~(\ref{flo-form})
and abbreviate $\Fl_{t,t_0}(y_0):=\Fl(t,t_0,y_0)$
for $t,t_0\in [a,b]$ and $y_0\in M$,
then
\begin{equation}\label{propaga}
\Fl_{t_2,t_1}(\Fl_{t_1,t_0}(y_0))=\Fl_{t_2,t_0}(y_0)\quad
\mbox{for all $t_0,t_1,t_2\in [a,b]$}
\end{equation}
holds as the left-hand side is
\begin{eqnarray*}
\sigma(\sigma(y_0,\eta(t_0)^{-1}\eta(t_1)),\eta(t_1)^{-1}\eta(t_2))
&=&\sigma(y_0,\eta(t_0)^{-1}\eta(t_1)\eta(t_1)^{-1}\eta(t_2))\\
&=&\sigma(y_0,\eta(t_0)^{-1}\eta(t_2))
\end{eqnarray*}
and thus equals the right-hand side.
Recall that
\begin{equation}\label{asmap}
\dot{\eta}=[t\mto \eta(t).\gamma(t)]=[t\mto T\lambda_{\eta(t)}\gamma(t)].
\end{equation}
Given $t_0\in[a,b]$ and $y_0\in M$,
let $\zeta(t):=\Fl_{t,t_0}(y_0)$ for $t\in [a,b]$.
Since $\zeta=\sigma_{y_0}\circ \lambda_{\eta(t_0)^{-1}}\circ\eta$,
the map $\zeta$ is absolutely continuous
and we have
\begin{eqnarray*}
\dot{\zeta}&=&[t\mto T\sigma_{y_0}T\lambda_{\eta(t_0)^{-1}}T\lambda_{\eta(t)}\gamma(t)]\\
&=&[t\mto T\sigma_{y_0}T\lambda_{\eta(t_0)^{-1}\eta(t)}\gamma(t)]\\
&=& [t\mto T\sigma_{\sigma(y_0,\eta(t_0)^{-1}\eta(t))}\gamma(t)]
=[t\mto T\sigma_{\zeta(t)}\gamma(t)]\\
&=& [t\mto\gamma(t)_\sharp(\zeta(t))],
\end{eqnarray*}
using \ref{chainRNEW} and (\ref{asmap}) for the first equality,
then (\ref{sta1}), and eventually~(\ref{sta0}).
Thus $\zeta$ is a Carath\'{e}odory solution to
the initial value problem
\begin{equation}\label{onceu}
\dot{y}(t)=\gamma(t)_\sharp(y(t)), \qquad y(t_0)=y_0.
\end{equation}
We define $f\colon [a,b]\times M\to TM$, $(t,y)\mto\gamma(t)_\sharp(y)$.
If $\phi\colon U_\phi\to V_\phi\sub E$ is a chart of~$M$,
consider the corresponding function
\[
f_\phi\colon [a,b]\times V_\phi\to E,\quad (t,z)\mto d\phi(f(t,\phi^{-1}(z))),
\]
as in Remark~\ref{mfdlocalize}.
Given $\bar{t}\in [a,b]$
and $\bar{z}\in V_\phi$, let $\bar{y}:=\phi^{-1}(\bar{z})$.
Let $X$ be the modelling space of~$G$
and $\psi\colon U_\psi\to V_\psi\sub X$
be a chart of~$G$ such that $\eta(\bar{t})\in U_\psi$.
There exists an open $\bar{y}$-neighbourhood $A\sub U_\phi$
and an open $\eta(\bar{t})$-neighbourhood $B\sub U_\psi$
such that
\[
\sigma(A\times B^{-1}B)\sub U_\phi.
\]
There exists an open $\bar{y}$-neighbourhood $K\sub A$
and an open $\eta(\bar{t})$-neighbourhood
$L\sub B$ such that
\[
\sigma(K\times L^{-1}L)\sub A.
\]
Then $V:=\phi(A)$ and $Y:=\phi(K)$ are open neighbourhoods
of $\bar{z}$ in~$V_\phi$ such that $Y\sub V$.
There exists a compact interval
$I\sub [a,b]$ which is a neighbourhood of~$\bar{t}$ in $[a,b]$,
such that $\eta(I)\sub L$.
We now show that
\[
\Phi\colon I\times I\times V\to V_\phi,\quad
(t,t_0,z_0)\mto \Phi_{t,t_0}(z_0):=\phi (\Fl_{t,t_0}(\phi^{-1}(z_0)))
\]
satisfies the conditions (a), (b), and (c)$'$
of Definition~\ref{defn-loc-flo-2} for $k=\infty$,
with~$f_\phi$ in place of~$f$
and $(\bar{t},\bar{z})$ in place of
$(\bar{t},\bar{y})$.
If this is true, then each $f_\phi$ (and hence $f$)
admits local flows which are pullbacks of $C^\infty$-maps,
whence $\dot{y}(t)=f(t,y(t))$ satisfies
local uniqueness
(see Lemma~\ref{secompl-flo})
and existence of Carath\'{e}odory
solutions, with maximal flow
the map $\Fl$ already introduced.\\[2.3mm]
For $t_0\in I$ and $z_0\in V$, the map
$I\to U_\phi$, $t\mto \Fl_{t,t_0}(\phi^{-1}(z_0))$
is a Carath\'{e}odory solution to (\ref{onceu})
with $y_0:=\phi^{-1}(z_0)$;
using \ref{chainRNEW},
we deduce that $I\to V_\phi$, $t\mto\Phi_{t,t_0}(z_0)$
is a Carath\'{e}odory solution to
$\dot{z}(t)=f_\phi(t,z(t))$, $z(t_0)=z_0$.
Thus condition~(a)
in Definition~\ref{defn-loc-flo-2} is satisfied,
and~(b) follows from~(\ref{propaga}). 
To verify condition~(c)$'$, we shall use the smooth mappings
\[
\Theta\colon L\times L\times A\to U_\phi,
\quad (g,h,y)\mto\sigma(y,h^{-1}g)
\]
and
\[
\Psi\colon \psi(L)\times \psi(L)\times V\to E,\quad (x,y,z)\mto
\phi(\Theta(\psi^{-1}(x),\psi^{-1}(y),\phi^{-1}(z))).
\]
Then $\Phi(t,t_0,z)=\Psi(\alpha(t),\beta(t_0),z)$
for all $(t,t_0,z)\in I\times I\times V$ using the absolutely continuous
functions $\alpha,\beta\colon I\to \psi(L)\sub X$,
\[
\alpha:=\psi\circ\eta|_I\quad\mbox{and}\quad
\beta:=\psi\circ\eta|_I.
\]
Let $q\colon E\to F$ be a continuous linear map to a Banach space~$F$
and $C\sub V$ be a compact subset.
By Lemma~\ref{factorCk},
there exist
continuous linear maps
$p_1\colon X\to Z_1$, $p_2\colon X\to Z_2$
and $p\colon E\to Z$ to Banach spaces
$Z_1$, $Z_2$, and $Z$,
open subsets $Q_1\sub Z_1$, $Q_2\sub Z_2$, and $Q\sub Z$,
open subsets $U_1\sub \psi(L)$, $U_2\sub \psi(L)$,
and $V_q\sub V$, and a $C^1$-function
\[
\Psi_q\colon Q_1\times Q_2\times Q\to F
\]
such that
\[
\alpha(I)\sub U_1,\quad \beta(I)\sub U_2,\quad
C\sub V_q,
\]
\[
p_1(U_1)\sub Q_1,\quad p_2(U_2)\sub Q_2,\quad p(V_q)\sub Q,
\]
and
\[
q(\Psi(x,y,z))=\Psi_q(p_1(x),p_2(y),p(z))\quad\mbox{for all
$\,(x,y,z)\in U_1\times U_2\times V_q$.}
\]
Note that $p_1\circ\alpha=p_1\circ\psi\circ \eta|_I$
is an absolutely continuous map to~$Z_1$ and
\[
(p_1\circ\alpha)'=[t\mto p_1\, d\psi \, T\lambda_{\eta(t)}\gamma(t)],
\]
using~\ref{chainRNEW}. Since $Z_1$ is a Banach space,
we therefore find a Borel set $I_q\sub I$
with $\lambda_1(I\setminus I_q)=0$
such that $(p_1\circ \alpha)'(t)$ exists for all
$t\in I_q$, and is given by
\[
(p_1\circ \alpha)'(t)=
p_1(d\psi \, T\lambda_{\eta(t)}\gamma(t)).
\]
Let $(t_0,z_0)\in I\times V_q$ and $y_0:=\phi^{-1}(z_0)$.
Then
\[
\kappa\colon I\to F,\quad
t\mto q(\Phi_{t,t_0}(z_0))
=\Psi_q(p_1(\alpha(t)),p_2(\beta(t_0)),p(z_0))
\]
is absolutely continuous
and differentiable at each $t\in I_q$ (see \ref{C1actsAC}), with
\begin{eqnarray*}
\kappa'(t)
&= &
d_1\Psi_q(p_1(\alpha(t)),p_2(\beta(t_0)),p(z_0);
p_1\, d \psi \, T\lambda_{\eta(t)}\gamma(t))\\
&=&
d_1(q\circ\Psi)(\alpha(t),\beta(t_0),z_0;d\psi \, T\lambda_{\eta(t)}\gamma(t))\\
&=&
d (q\circ \phi\circ \Theta(\cdot,\eta(t_0),y_0))T\lambda_{\eta(t)}\gamma(t)\\
&=&
q \,d\phi \, T\sigma_{\sigma(y_0,\eta(t_0)^{-1})}T\lambda_{\eta(t)}\gamma(t)\\
&=& q \, d\phi \, \gamma(t)_\sharp(\phi^{-1}(\Phi_{t,t_0}(z_0)))
\;=\,
q(f_\phi(t,\Phi_{t,t_0}(z_0)))
\end{eqnarray*}
at $t\in I_q$,
using that
$\Theta(\cdot,\eta(t_0),y_0)=\sigma(y_0,\eta(t_0)^{-1}\cdot)
=\sigma_{\sigma(y_0,\eta(t_0)^{-1})}$
and
\begin{eqnarray*}
T\sigma_{\sigma(y_0,\eta(t_0)^{-1})} T\lambda_{\eta(t)}\gamma(t)
&=&T\sigma_{\sigma(y_0,\eta(t_0)^{-1}\eta(t))}\gamma(t)=\gamma(t)_\sharp(\Fl_{t,t_0}(y_0))\\
&=&\gamma(t)_\sharp(\phi^{-1}(\Phi_{t,t_0}(z_0))).
\end{eqnarray*}
Thus condition~(c)$'$ is verified.
\end{proof}
\begin{rem}\label{givesC}
Lemma~\ref{approx-stair} remains valid if,
more generally, $E$ is a locally convex space
and $\gamma\colon [a,b]\to E$ a Lusin measurable
$\cL^1$-function.\\[2.3mm]
[To see this, let $(K_n)_{n\in\N}$
be a sequence of compact subsets $K_n\sub [a,b]$
such that $\gamma|_{K_n}$ is continuous
and $A:=\bigcup_{n\in\N}K_n$ satisfies $\lambda_1([a,b]\setminus A)=0$.
Let $c\in\gamma([a,b])$.
If we define $\bar{\gamma}(t):=\gamma(t)$ if $t\in A$,
$\bar{\gamma}(t):=c$ if $t\in[a,b]\setminus A$,
then $\bar{\gamma}\colon [a,b]\to E$
is Lusin measurable and Borel measurable,
and $\gamma(t)=\bar{\gamma}(t)$ almost everywhere.
In the proof, we replace $\gamma$ with $\bar{\gamma}$;
then $\bar{\gamma}-\gamma_1$, $\gamma_1-\gamma_2$,
$\gamma_2-\eta$, and $\eta-\theta_k$
are Borel measurable by Lemma~\ref{compo-prod}.
Eventually, $\|\gamma-\eta\|_{\cL^1,q}\leq \|\gamma-\bar{\gamma}\|_{\cL^1,q}
+\|\bar{\gamma}-\eta\|_{\cL^1,q}=\|\bar{\gamma}-\eta\|_{\cL^1,q}\leq\ve$
as $\|\gamma-\bar{\gamma}\|_{\cL^1,q}=0$. Likewise,
$\|\bar{\gamma}-\theta\|_{\cL^1,\ve}\leq\ve$
implies $\|\gamma-\theta\|_{\cL^1,q}\leq\ve$.]\\[2.3mm]
Having generalized Lemma~\ref{approx-stair},
the proof of Theorem~\ref{thmC}
remains valid for~$G$ and~$M$ modelled on sclc-spaces.
\end{rem}
\begin{rem}\label{givesDE}
Theorem~\ref{thmD}
remains valid if $G$ is an $L^p$-regular
Lie group modelled on an sclc-space;
simply replace $\cE$ with $L^p$
in the proof of Theorem~\ref{L1cont};
the subdivision property
for $L^p$-spaces based on Lusin measurability
is provided in~\cite[Lemma~4.3.12]{Nik}.
Having generalized Theorem~\ref{thmD},
the proof of Theorem~\ref{thmE}
is unchanged for~$G$ and~$M$
modelled on sclc-spaces.
\end{rem}
\begin{rem}\label{furth-var}
(a) In \cite[1.40]{MeR}, $L^p$-spaces associated with Borel
measurable $E$-valued
functions have also been considered
if an sclc-space~$E$
has the Fr\'{e}chet exhaustion property
(recalled in \ref{theFEP}),
as well as $L^p$-regularity
for Lie groups modelled on such spaces.
As observed in \cite[Remark~4.1.16]{Nik}, these $L^p$-spaces
are isomorphic to those based on Lusin measurable
maps (and $L^p$-regularity as in \cite{MeR}
is then equivalent to that in~\cite{Nik}).
We therefore need not develop further variants
of our current results for (FEP)-spaces,
they are covered by the Lusin theory.\medskip

\noindent
(b) If $E$ is a locally convex space, we recalled
in \ref{theLRC} the vector spaces $\cL^\infty_{rc}([a,b],E)$
and $L^\infty_{rc}([a,b],E)$
introduced in~\cite{Mea}.
Any such $\gamma$ is the uniform limit of a sequence
of measurable functions with finite image (see \cite[Proposition~3.18]{Mea}),
and it is Lusin measurable;
we can interpret $L^\infty_{rc}([a,b],E)$
with a vector subspace of $L^\infty([a,b],E)$
in the Lusin sense. Now the point is that
sequential completeness of~$E$
is unnecessary to ensure that
primitives $\eta\colon [a,b]\to E$,
\[
\eta(t):= c+\int_a^t\gamma(s)\,ds
\]
(which we call $AC_{L^\infty_{rc}}$-functions)
exist for $\gamma\in \cL^\infty_{rc}([a,b],E)$; they exist whenever~$E$
is integral complete in the sense recalled in~\ref{intecompl}.\\[2.3mm]
An $\cL^\infty_{rc}$-function $\gamma\colon [a,b]\to E$
is called \emph{regulated} if $\gamma$ is the uniform limit
of a sequence of $E$-valued staircase functions on~$[a,b]$;
following~\cite[1.31]{MeR}, we let $\cR([a,b],E)$
be the vector space of all
such functions and $R([a,b],E)\sub L^\infty_{rc}([a,b],E)$
be the corresponding vector subspace.
Note that $\cR([a,b],E)$ contains all
piecewise continuous functions,
continuous functions, and staircase functions.
See \cite{MeR} for corresponding concepts
of $L^\infty_{rc}$-regular Lie groups and $R$-regular
Lie groups modelled on integral complete locally convex spaces.\\[2.3mm]
All of our results for $L^p$ with $p=\infty$
and corresponding absolutely continuous functions
remain valid for integral complete locally convex spaces
and modelling spaces in place of sclc-spaces,
by straightforward adaptations.
Notably, $\cE=L^\infty_{rc}$ and $\cE=R$
are possible choices for $\cE$ in Theorem~\ref{L1cont}
(as it stands).
\end{rem}
In Theorems~\ref{thmB}
and~\ref{thmBgen}
which are the basis for our work,
the maximal flow is globally defined,
on all of $[a,b]\times [a,b]\times M$.
We mention two basic general facts concerning
maximal flows, which may be of interest elsewhere.
\begin{numba}\label{the-situ-nw}
Let $M$ be a $C^1$-manifold modelled on an sclc-space,
$J\sub\R$ be a non-degenerate interval
and $f\colon W\to TM$ be a function on a subset
$W\sub J\times M$
such that $f(t,y)\in T_yM$ for all $(t,y)\in W$.
Assume that the differential equation {\rm(\ref{ode-mfd})}
satisfies both local existence of Carath\'{e}odory
solutions and local uniqueness.
Let $\Omega\sub J\times J\times M$,
the maximal flow $\Fl\colon \Omega\to M$,
and $\gamma_{t_0,y_0}\colon I_{t_0,y_0}\to M$ for $(t_0,y_0)\in W$
be as in Definition~\ref{max-flow}
(cf.\ also Remark~\ref{all-holds}).
For $t,t_0\in J$, we define
\[
\Omega_{t,t_0}:=\{y_0\in M\colon (t,t_0,y_0)\in\Omega\}
\]
and abbreviate $\Fl_{t,t_0}(y_0):= \Fl(t,t_0,y_0)$
for $y_0\in\Omega_{t,t_0}$.
\end{numba}
The following fact can be proved like
Lemmas 2.4.9 and 2.5.7 in~\cite{GaN}.
\begin{la}
In the situation of {\rm\ref{the-situ-nw}}, we have:
\begin{itemize}
\item[\rm(a)]
For $(t_0,y_0)\in W$ and $t_1 \in
I_{t_0,y_0}$, we have $\gamma_{t_0,y_0} = \gamma_{t_1,y_1}$
with $y_1 := \gamma_{t_0,y_0}(t_1) = \Fl_{t_1,t_0}(y_0)$.
\item[\rm(b)]
If $t_2 \in I_{t_1,y_1}$ in {\rm(a)}, then $t_2 \in I_{t_0,y_0}$
and $\Fl_{t_2,t_0}(y_0) = \Fl_{t_2,t_1}(\Fl_{t_1,t_0}(y_0))$.
\item[\rm(c)]
For all $t_1,t_0 \in J$, the map $\Fl_{t_1,t_0}\colon \Omega_{t_1,t_0}\to M$ is injective,
$\Fl_{t_1,t_0}(\Omega_{t_1,t_0}) = \Omega_{t_0,t_1}$ holds,
and $\Fl_{t_0,t_1} = (\Fl_{t_1,t_0})^{-1}$. $\,\square$
\end{itemize}
\end{la}
\begin{defn}
Let $J\sub\R$ be a non-degenerate interval,
$M$ be a $C^1$-manifold modelled on an sclc-space
and $f\colon W\to TM$ be a function on an open subset $W \sub J \times M$
such that $f(t,y)\in T_yM$ for all $(t,y) \in W$.
We say that the differential equation $\dot{y}(t) = f(t,y(t))$
\emph{admits local $C^0$-flows in the sense of Carath\'{e}odory
solutions}
if, for all $(\bar{t},\bar{y}) \in W$, there exist a relatively open interval $I \sub J$
with $\bar{t} \in I$, an open $\bar{y}$-neighbourhood~$V$
in $M$ and a continuous function $\Phi\colon I \times I \times V \to M$
which satisfies conditions~(a) and~(b)
stated in Definition~\ref{def-loc-flo-mfd}.
\end{defn}
The following fact can be proved like
Theorem~2.5.15 in \cite{GaN}
(taking a singleton set of parameters
in \cite[Theorem~2.5.18]{GaN}
and reading `solution' as
`Carath\'{e}odory solution').
\begin{prop}
Let $J \sub\R$ be a non-degenerate interval,
$M$ be a $C^1$-manifold modelled on an sclc-space
and $f \colon W \to TM$ be a function on an open subset $W \sub J\times M$
such that $f(t,y) \in T_yM$ for all $(t,y) \in W$.
Assume that $\dot{y}(t) = f(t,y(t))$ satisfies local
uniqueness of Carath\'{e}odory solutions and
admits local $C^0$-flows
in the sense of Carath\'{e}odory solutions.
Let $\Fl\colon \Omega\to M$ be the maximal flow of
$\dot{y}(t) = f(t,y(t))$. Then $\Omega$ is open in $J\times J\times M$
and $\Fl\colon \Omega\to M$ is continuous.
Moreover, $\Omega_{t,t_0}$
is open in $M$ for all $t,t_0\in J$
and the map $\Fl_{t,t_0}\colon \Omega_{t,t_0}\to\Omega_{t_0,t}$
is a homeomorphism, with inverse $\Fl_{t_0,t}$. $\,\square$
\end{prop}
\appendix
\section{Local factorization of {\boldmath$C^{k+1}$}-mappings over Banach spaces}\label{appA}
The following lemma is used in the proof of
Lemma~\ref{secompl-flo}.
See~\cite{Cha}
for a precursor in the case $k=1$;
cf.\
also
\cite[\S7]{Hyp} and \cite[Lemma~1.62]{MeR}
for related results.
\begin{la}\label{factorCk}
Let $E$ be a locally convex space,
$(Z,\|\cdot\|)$ be a Banach space,
$U\sub E$ be open, $K\sub U$ a compact subset,
$k\in\N_0$
and $f\colon U\to Z$ be a $C^{k+1}$-map.
Then the following holds:
\begin{itemize}
\item[\rm(a)]
There exist a continuous linear map
$q\colon E\to F$ to a Banach space~$F$,
a $C^k$-function
\[
g\colon Q\to Z
\]
on an open subset $Q\sub F$,
and
an open subset $V\sub U$ with $q(V)\sub Q$
such that $K\sub V$ and
\[
g\circ q|_V=f|_V.
\]
\item[\rm(b)]
If $E=E_1\times \cdots\times E_m$
with locally convex spaces $E_1,\ldots, E_m$
and $U=U_1\times\cdots\times U_m$
with open subsets $U_j\sub E_j$
for $j\in\{1,\ldots,m\}$,
we can achieve in {\rm(a)}
that $V=V_1\times\cdots\times V_m$
with open subsets $V_j\sub U_j$,
$F=F_1\times\cdots\times F_m$
with Banach spaces $F_j$, $Q=Q_1\times\cdots\times Q_m$
with open subsets $Q_j\sub F_j$
and $q=q_1\times\cdots\times q_m$
with continuous linear maps $q_j\colon E_j\to F_j$ such that
$q_j(V_j)\sub Q_j$, for all $j\in\{1,\ldots, m\}$.
\end{itemize}
\end{la}
The following concepts are useful for the proof.
\begin{numba}\label{associated-ban}
If $E$ is a locally convex space
and $p$ a continuous seminorm on~$E$,
then $N:=\{x\in E\colon p(x)=0\}$
is a vector subspace of~$E$ and $\|x+N\|_p:=p(x)$
is well defined for $x+N\in E_p:=E/N$
and provides a norm $\|\cdot\|_p$ on $E_p$.
If $\pi_p\colon E\to E_p$, $x\mto x+N$ is the canonical map,
then $\|\pi_p(x)\|_p=p(x)$ for all $x\in E$,
entailing that
\begin{equation}\label{image-ball}
\pi_p(B^p_\ve(x))=B^{\|\cdot\|_p}_\ve(\pi_p(x))\mbox{ for all $x\in E$
and $\ve>0$.}
\end{equation}
We let $\wt{E}_p$ be a completion of
$E_p$ with $E_p\sub\wt{E}_p$
and use the same notation, $\|\cdot\|_p$,
for the norm on $\wt{E}_p$.
We write $\wt{\pi}_p$ for $\pi_p$,
considered as a map to $\wt{E}_p$.
\end{numba}
Lemma~\ref{factorCk}\,(a) follows immediately
from the next lemma, which we prove instead.
And also Lemma~\ref{factorCk}\,(b)
follows:
If $E=E_1\times\cdots\times E_m$
and $U=U_1\times\cdots\times U_m$,
we let $\pr_j\colon E\to E_j$ be the projection onto the $j$th
component for $j\in\{1,\ldots, m\}$ and set $K_j:=\pr_j(K)\sub U_j$.
After increasing~$K$, we may assume that $K=K_1\times\cdots\times K_m$.
For~$P$ as in Lemma~\ref{easierfactor},
there exists a continuous seminorm $p\geq P$ on~$E$ satisfying
\[
p(x_1,\ldots, x_m)=\max\{p_1(x_1),\ldots,p_m(x_m)\}
\mbox{ for all $(x_1,\ldots, x_m)\in E$}
\]
with continuous seminorms $p_j$ on~$E_j$
(for example, we can take $p_j(x_j):=$ $mP(0,\ldots,0,x_j,0,\ldots,0)$
with $x_j$ in the $j$th slot). For any such, we have
\[
\wt{E}_p\cong \wt{E}_{p_1}\times\cdots\times\wt{E}_{p_m}
\]
in a natural way and $Q_p:=\tilde{\pi}_p(K)+B^{\tilde{E}_p}_1(0)$
corresponds to $Q_1\times\cdots\times Q_m$ with
$Q_j:=\tilde{\pi}_{p_j}(K_j)+B^{\tilde{E}_{p_j}}_1(0)$.
\begin{la}\label{easierfactor}
Let $E$ be a locally convex space,
$(Z,\|\cdot\|)$ be a Banach space,
$U\sub E$ be open, $K\sub U$ a compact subset,
$k\in\N_0$
and $f\colon U\to Z$ be a $C^{k+1}$-map.
Then there exists a continuous
seminorm $P$ on~$E$
with $K+B^P_1(0)\sub U$
with the following property:
For each continuous seminorm $p$ on~$E$
such that $p\geq P$ pointwise,
there exists a $C^k$-map
\[
g\colon Q_p\to Z
\]
on the open subset $Q_p:=\tilde{\pi}_p(K)+B^{\wt{E}_p}_1(0)$
of $\wt{E}_p$ such that $g\circ \tilde{\pi}_p|_{V_p}=f|_{V_p}$
holds for the open subset $V_p:=K+B^p_1(0)\sub U$
with $K\sub V_p$.
\end{la}
{\bf Proof of Lemma~\ref{easierfactor}.}
The proof is by induction on~$k$.
The case~$k=0$:
We assume that~$f$ is~$C^1$.
There exists $M\in\;]0,\infty[$
such that $\sup \|f(K)\|<M$.
Let $B\sub Z$ be the open unit ball.
As $df\colon U\times E\to Z$ is continuous and $df(x,0)=0$ for all $x\in U$,
the set $(df)^{-1}(B)$ is an open neighbourhood of $K\times \{0\}$ in $U\times E$.
By the Wallace Lemma, there exist an open subset $O\sub U$ with $K\sub O$
and an open $0$-neighbourhood $W\sub E$ such that $O\times W\sub (df)^{-1}(B)$.
After shrinking~$O$, we may assume that
\begin{equation}\label{upbound}
\|f(x)\|<M\mbox{ for all $x\in O$.}
\end{equation}
After shrinking~$W$, we may assume that $K+2W\sub O$
and $W=B^P_1(0)$ for a continuous seminorm~$P$ on~$E$.
Then
\[
\|df(x,y)\|\leq P(y) \mbox{ for all $x\in O$ and $y\in E$.}
\]
We now define $V:=K+W$.
Let $y,z\in V$. If $z-y\in W$,
pick $x\in K$ such that $y-x\in W$. Then
$z-x=(z-y)+(y-x)\in 2W$ and thus $y,z\in x+2W$,
which is a convex set contained in~$O$.
By the Mean Value Theorem,
\begin{eqnarray*}
\|f(z)-f(y)\|&=& \left\|\int_0^1 df(y+t(z-y),z-y)\,dt\right\|\\
&\leq & \int_0^1\|df(y+t(z-y),z-y)\|\,dt\leq P(z-y).
\end{eqnarray*}
If $z-y\not\in W$, then $P(z-y)\geq 1$
and thus
\[
\|f(z)-f(y)\|\leq\|f(z)\|+\|f(y)\|\leq 2M\leq 2M\, P(z-y).
\]
Thus, after replacing $P$ with $\max\{1,2M\}\,P$
(and shrinking~$W$ and~$V$ accordingly),
we may assume
$\|f(z)-f(y)\|\leq P(z-y)$ for all $y,z\in V$.
Hence
\begin{equation}\label{lip-around-cp}
\|f(z)-f(y)\|\leq p(z-y)
\mbox{ for all $y,z\in V_p$}
\end{equation}
for each continuous seminorm $p$ on $E$ such that $p\geq P$
pointwise, with $V_p:=K+B^p_1(0)$.
Note that
\[
\pi_p(V_p)=\pi_p(K)+B^{E_p}_1(0)
\]
is open in $E_p$. If $y,z\in V_p$ such that $\pi_p(y)=\pi_p(z)$
and thus $p(z-y)=0$, we have $f(y)=f(z)$ by (\ref{lip-around-cp}),
showing that
\[
h_p\colon \pi_p(V_p)\to Z,\quad \pi_p(y)\mto f(y)
\]
is well defined. By construction, $h_p\circ\pi_p|_{V_p}=f|_{V_p}$.
Note that $\pi_p(V_p)=\pi_p(K)+B^{E_p}_1(0)$ is dense in
\[
Q_p:=\pi_p(K)+B^{\wt{E}_p}_1(0).
\]
Moreover, $h_p$ is Lipschitz continuous (and thus uniformly
continuous), since
\[
\|h_p(\pi_p(z))-h_p(\pi_p(y))\|=\|f(z)-f(y)\|\leq p(z-y)=\|\pi_p(z)-\pi_p(y)\|_p
\]
for all $y,z\in V_p$. As $(Z,\|\cdot\|)$
is complete, the uniformly continuous map $h_p$ has a unique
uniformly continuous
extension
\[
g_p\colon Q_p\to Z.
\]
Then $g_p\circ\wt{\pi}_p|_{V_p}=h_p\circ \pi_p|_{V_p}=f|_{V_p}$.\\[2.3mm]
To perform the
induction step, let $f$ be $C^{k+2}$
and assume that the assertion of the lemma holds
for~$k$. Let $P$, $V_p$, and $g_p\colon Q_p\to Z$ for $p\geq P$ be as in
the step $k=0$. The map $df\colon U\times E\to Z$ is $C^{k+1}$
and $K\times\{0\}$ a compact subset of $U\times E$.
By the inductive hypothesis, there is a continuous
seminorm $R$ on $E\times E$
such that $(K\times\{0\})+B^R_1(0)\sub U\times E$
and there exists a $C^k$-map
\[
h_r\colon S_r\to Z
\]
on $S_r:= \tilde{\pi}_r(K\times \{0\})+B^{(E\times E)_r^{\wt{\;}}}_1(0)$
such that $h_r\circ \tilde{\pi}_r|_{W_r}=df|_{W_r}$
on $W_r:=(K\times\{0\})+B^r_1(0)$,
for each continuous seminorm $r$ on $E\times E$ such that $r\geq R$.
Note that
\[
p(x):=\max\{P(x),2R(x,0),2R(0,x)\}
\]
defines a continuous seminorm on~$E$ such that $p\geq P$.
Moreover, $r(x,y):=\max\{p(x),p(y)\}$ defines a continuous
seminorm on $E\times E$ with $r\geq R$.
Then
\[
W_r=V_p\times B^p_1(0)\quad\mbox{and}\quad
S_r=Q_p\times B^{\tilde{E}_p}_1(0),
\]
identifying $(E\times E)_r^{\wt{\;}}$
with $\tilde{E}_p\times \tilde{E}_p$.\\[2.3mm]
We now show that the directional
derivative $dg_p(x,y)$
exists for all $x\in Q_p\cap E_p$
and $y\in B^{E_p}_1(0)$, and is given by
\[
dg_p(x,y)= h_r(x,y).
\]
There exists $z\in K$ such that
\[
x\in E_p\cap (\tilde{\pi}_p(z)+B^{\tilde{E}_p}_1(0))
=E_p\cap B^{\tilde{E}_p}_1(\tilde{\pi}_p(z))
=B^{E_p}_1(\pi_p(z)).
\]
Let $\ve>0$ such that $\|x-\pi_p(z)\|_p+\ve < 1$.
Let $a\in B^p_1(z)$ such that $\pi_p(a)=x$
and $b\in B^p_1(0)$ such that $\pi_p(b)=y$.
Then $p(a+s\ve b-z)\leq p(a-z)+\ve p(b)=\|x-\pi_p(z)\|_p+\ve\|y\|_p<1$
for all $s\in [{-1},1]$,
whence $a+s \ve b\in B^p_1(z)$. For $0\not=t\in \,]{-\ve},\ve[$,
we get
\begin{eqnarray*}
\frac{g_p(x+ty)-g_p(x)}{t}
&=& \frac{1}{\ve}\frac{g_p(x+(t/\ve)(\ve y))-g_p(x)}{(t/\ve)}\\
&=&
\frac{1}{\ve}\frac{f(a+(t/\ve)(\ve b))-f(a)}{(t/\ve)}\\
&\to &\frac{1}{\ve}\, df(a,\ve b)= df(a,b)=
h_r(x,y)
\end{eqnarray*}
as $t\to 0$, as we set out to show.\\[2.3mm]
For all $(x,y)\in (Q_p\cap E_p)\times E_p$
and $s>0$ such that $y\in sB^{E_p}_1(0)$,
the directional derivative
$dg_p(x,s^{-1}y)$ exists by the preceding, whence
also $dg_p(x,y)=dg_p(x,s(s^{-1}y))$ exists by a standard argument,
and is given by
\begin{equation}\label{cover-s}
dg_p(x,y)=sdg_p(x,s^{-1}y)=sh_r(x,s^{-1}y).
\end{equation}
Thus $d(g_p|_{Q_p\cap E_p}) \colon (Q_p\cap E_p)\times E_p\to Z$
exists and is a $C^k$-map, since the sets $U_s:=(Q_p\cap E_p)\times sB^{E_p}_1(0)$
form an open cover of $(Q_p\cap E_p)\times E_p$
for $s\in\;]0,\infty[$
and $d(g_p|_{Q_p\cap E_p})(x,y)$ is $C^k$ in $(x,y)\in U_s$,
being given by~(\ref{cover-s}).
Thus $g_p|_{Q_p\cap E_p}$ is~$C^{k+1}$ and thus~$C^1$.\\[2.3mm]
For each $x_0\in Q_p$, there exists $\rho>0$
such that $B^{\tilde{E}_p}_{2\rho}(x_0)\sub Q_p$.
Let
$\Omega:=B^{\tilde{E}_p}_\rho(x_0)$.
We show that $g_p|_\Omega$
is~$C^{k+1}$; as such sets~$\Omega$ form an open cover of~$Q_p$,
this implies that~$g_p$ is~$C^{k+1}$.
The function
\[
\Delta\colon \Omega\times B^{\tilde{E}_p}_1(0)\times
\,]{-\rho},\rho[\,\to Z,\quad
(x,y,t)\mto \int_0^1
h_r(x+sty,y)\, ds
\]
is continuous, being a parameter-dependent integral
(see \cite[Lemma~1.1.11]{GaN}).
For $(x,y)\in (\Omega\cap E_p)\times B^{E_p}_1(0)$
and $0\not=t\in \,]{-\rho},\rho[$,
we have
\begin{equation}\label{both-sids}
\Delta(x,y,t)=\int_0^1d(g_p|_{Q_p\cap E_p})(x+sty,y)\,ds
=\frac{g_p(x+ty)-g(x)}{t},
\end{equation}
by the Mean Value Theorem.
Note that the right-hand side
of (\ref{both-sids}) defines
a continuous $Z$-valued function on $\Omega\times B^{\tilde{E}_p}_1(0)
\times (]{-\rho},\rho[\,\setminus\{0\})$,
in which
$(\Omega\cap E_p)\times B^{E_p}_1(0)\times (]{-\rho},\rho[\,\setminus\{0\})$
is dense. As also the left-hand side defines a continuous function there
and~$Z$ is Hausdorff, we deduce that
\[
\Delta(x,y,t)=\frac{g_p(x+ty)-g(x)}{t}
\mbox{ for all $\Omega\times B^{\tilde{E}_p}_1(0)
\times (]{-\rho},\rho[\,\setminus\{0\})$.}
\]
Letting $t\to 0$, we see that $dg_p(x,y)$ exists for all
$x\in\Omega$ and $y\in B^{\tilde{E}_p}_1(0)$,
and is given by
\[
dg_p(x,y)=\Delta(x,y,0)=h_r(x,y).
\]
As above, we deduce from this that $g_p|_\Omega$
as $C^{k+1}$. $\,\square$\\[3.5mm]
{\bf Acknowledgements.}
The research was supported by Deutsche Forschungsgemeinschaft (DFG),
project GL 357/9-1.
The authors thank the anonymous referee,
whose comments helped to improve the presentation.
{\bf Helge  Gl\"{o}ckner}, Institut f\"{u}r Mathematik, Universit\"at Paderborn,\\
Warburger Str.\ 100, 33098 Paderborn, Germany; {\tt glockner@math.upb.de}\\[2.5mm]
{\bf Joachim Hilgert}, Institut f\"{u}r Mathematik, Universit\"at Paderborn,\\
Warburger Str.\ 100, 33098 Paderborn, Germany; {\tt hilgert@math.upb.de}\vfill

\begin{thebibliography}{99}\itemsep+.6pt
%
%
\bibitem{AaS}
Agrachev, A.\,A. and Y.\,L. Sachkov,
``Control Theory from the Geometric Viewpoint,''
Spinger, Berlin, 2004.
%
%
\bibitem{Bas}
Bastiani, A.,
\emph{Applications diff\'{e}rentiables et vari\'{e}t\'{e}s diff\'{e}rentiables de dimension infinie},
J. Anal.\ Math.\ {\bf 13} (1964), 1--114.
%
%
\bibitem{BGN}
Bertram, W., H. Gl\"{o}ckner, and
K.-H. Neeb,
\emph{Differential calculus over general base fields and rings},
Expo.\ Math.\ {\bf 22} (2004), 213--282.
%
%
\bibitem{Boc}
Bochner, S.,
\emph{Absolut-additive abstrakte Mengenfunktionen},
Fundam.\ Math.\ {\bf 21} (1933), 211--213. 
%
%
\bibitem{Car}
Cartan, H., ``Calcul
diff\'{e}rentiel,''
Hermann, Paris 1967.
%
%
\bibitem{Cha}
Chasiotis, A.-V., ``Vektorwertige absolut stetige Funktionen,''
Bachelor's thesis, University of Paderborn, 2020
(advised by H. Gl\"{o}ckner).
%
%
\bibitem{Cla}
Clarkson, J.\,A.,
\emph{Uniformly convex spaces},
Trans.\ Amer.\ Math.\ Soc.\
{\bf 40} (1936), 396--414.
%
%
\bibitem{Die}
Dieudonn\'{e}, J.,
``Foundations of Modern Analysis,''
Academic Press, New York, 1969.
%
%
\bibitem{Eng}
Engelking, R., ``General Topology'',
Heldermann Verlag, Berlin, 1989.
%
%
\bibitem{Eyn}
Eyni, J.\,M.,
\emph{The Frobenius theorem for Banach distributions on infinite-dimensional manifolds and
applications in infinite-dimensional Lie theory},
preprint, arXiv:1407.3166.
%
%
\bibitem{FMP}
Florencio, M.,
F. Mayoral, and P.\,J. Pa\'{u}l,
\emph{Spaces of vector-valued integrable functions and localization of bounded subsets},
Math.\ Nachr.\ {\bf 174} (1995), 89--111.
%
%
\bibitem{Res}
Gl\"{o}ckner, H.,
\emph{Infinite-dimensional Lie groups without completeness restrictions},
pp.\ 43--59 in: Strasburger, A. et al.\ (eds.),
``Geometry and Analysis on Finite- and Infinite-Dimensional Lie Groups,''
Banach Center Publications {\bf 55}, Warsaw, 2002.
%
%
\bibitem{GCX}
Gl\"{o}ckner, H.,
\emph{Lie group structures on quotient groups and
universal complexifications for infinite-dimensional Lie groups},
J. Funct.\ Anal.\ {\bf 194} (2002), 347--409.
%
%
\bibitem{Mea}
Gl\"{o}ckner, H.,
\emph{Lie groups of measurable mappings},
Can.\ J. Math.\ {\bf 55} (2003), 969--999.
%
%
\bibitem{DIR}
Glöckner, H.,
\emph{Fundamentals of direct limit Lie theory},
Compositio Math.\ {\bf 141} (2005),
1551--1577.
%
%
\bibitem{Hyp}
Gl\"{o}ckner, H.,
\emph{Aspects of differential calculus related to infinite-dimensional
vector bundles and Poisson vector spaces}, Axioms {\bf 2022}, 11, 221.
%
%
\bibitem{SEM}
Gl\"{o}ckner, H.,
\emph{Regularity properties of infinite-dimensional Lie groups,
and semiregularity}, preprint, arXiv:1208.0715.
%
%
\bibitem{MeR}
Gl\"{o}ckner, H.,
\emph{Measurable regularity properties
of infinite-dimensional Lie groups},
preprint, arXiv:1601.02568.
%
%
\bibitem{Ana}
Gl\"{o}ckner, H.,
\emph{Lie groups of real analytic diffeomorphisms
are $L^1$-regular},
preprint, arXiv:2007.15611.
%
%
\bibitem{GaN} Gl\"{o}ckner, H. and K.-H. Neeb,
``Infinite-Dimensional Lie Groups,'' book in preparation.
%
% 
\bibitem{Ham}
Hamilton, R.\,S.,
\emph{The inverse function theorem of Nash and Moser},
Bull.\ Amer.\ Math.\ Soc.\ {\bf 7} (1982), 65--222.
%
%
\bibitem{Tro}
Hanusch, M.,
\emph{The strong Trotter property for locally $\mu$-convex Lie groups},
J. Lie Theory {\bf 30} (2020), 25--32.
%
%
\bibitem{Han}
Hanusch, M.,
\emph{Regularity of Lie groups},
Commun.\ Anal.\ Geom.\ {\bf 30} (2022),
53--152.
%
%
\bibitem{HHL}
Hilgert, J., K.\,H. Hofmann, and J.\,D. Lawson,
``Lie Groups, Convex Cones, and Semigroups,''
Clarendon Press, Oxford, 1989.
%
%
\bibitem{Jur}
Jurdjevic, V.,
``Geometric Control Theory,''
Cambridge Univ.\ Press, 1997.
%
%
\bibitem{JaS}
Jurdjevic, V. and H.\,J. Sussmann,
``Control Systems on Lie Groups,''
J. Differ.\ Equations {\bf 12} (1972), 313--329. 
%
%
\bibitem{Kel}
Keller, H.\,H., ``Differential Calculus
in Locally Convex Spaces'', Springer,
Berlin, 1974.
%
%
\bibitem{KM0}
Kriegl, A. and P.\,W. Michor,
\emph{Regular infinite dimensional Lie groups},
J. Lie Theory {\bf 7} (1997), 61--99.
%
%
\bibitem{KaM} Kriegl, A. and P.\,W. Michor,
``The Convenient Setting of Global Ana\-lysis,''
AMS, Providence, 1997.
%
%
\bibitem{Lan}
Lang, S.,
``Fundamentals of Differential Geometry,''
Springer, New York, 1999.
%
%
\bibitem{Mic} Michor, P.\,W., ``Manifolds of Differentiable Mappings,''
Shiva Publ., Orpington, 1980.
%
%
\bibitem{Mi0}
Milnor, J.,
\emph{On infinite-dimensional Lie groups}, preprint, Institute for
Advanced Study, Princeton, 1982.
%
%
\bibitem{Mil} Milnor, J., \emph{Remarks on infinite-dimensional Lie groups},
pp.\,1007--1057 in: B.\,S. DeWitt and R. Stora (eds.),
``Relativit\'{e}, groupes et topologie~II,'' North-Holland,
Amsterdam, 1984.
%
%
\bibitem{Nee}
Neeb, K.-H.,
\emph{Towards a Lie theory of locally convex groups},
Jpn.\ J. Math.\ {\bf 1} (2006), 291--468.
%
%
\bibitem{Nik}
Nikitin, N., ``Regularity Properties of Infinite-Dimensional Lie Groups
and Exponential Laws,'' doctoral dissertation,
University of Paderborn,
2021; see {\tt nbn-resolving.de/urn:nbn:de:hbz:466:2-39133}
%
%
\bibitem{Rud}
Rudin, W.,
``Real and Complex Analysis,''
McGraw-Hill, New York, 1987.
%
%
\bibitem{Sac}
Sachkov, Y.\,L.,
\emph{Control theory on Lie groups},
J. Math.\ Sci., New York {\bf 156} (2009), 381--439.
%
% 
\bibitem{Sch}
Schechter, E., ``Handbook of Analysis and its Foundations,''
Academic Press, San Diego, 1997.
%
%
\bibitem{Son}
Sontag, E.\,D.,
``Mathematical Control Theory,''
Springer, New York, ${}^2$1998. 
%
%
\bibitem{Voi}
Voigt, J.,
\emph{On the convex compactness property for the strong operator topology},
Note Mat.\ {\bf 12} (1992), 259--269.
%
%
\bibitem{Wal}
Walter, B.,
\emph{Weighted diffeomorphism groups of Banach spaces and weighted mapping groups},
Diss.\ Math.\ {\bf 484} (2012), 126 pp. 
%
%
\bibitem{Wei}
Weizs\"{a}cker, H. von,
\emph{In which spaces is every curve Lebesgue-Pettis integrable}?,
preprint, arXiv:1207.6034.\vspace{1.8mm}
%
%
\end{thebibliography}
\end{document}